\documentclass[11pt]{article}
\usepackage{amsfonts, epsfig, amsmath, amssymb, color}
\usepackage{textcomp}
\usepackage{paralist}
\usepackage{selinput}
\usepackage{ulem}

\usepackage[linkcolor=black,citecolor=black]{hyperref}

\usepackage[english]{babel}

\renewcommand {\epsilon}{\varepsilon}

\newcommand{\CC}{\mathbb{C}}

\newcommand{\EE}{\mathbb{E}}

\newcommand{\NN}{\mathbb{N}}

\newcommand{\PP}{\mathbb{P}}

\newcommand{\RR}{\mathbb{R}}

\newcommand{\bB}{\mathcal{B}}
\newcommand{\cC}{\mathcal{C}}

\newcommand{\fF}{\mathcal{F}}

\newcommand{\ffF}{\mathfrak{F}}

\newcommand{\fK}{\mathfrak{K}}

\newcommand{\fM}{\mathfrak{M}}

\newcommand{\al}{\alpha}

\newcommand{\e}{\varepsilon}
\newcommand{\la}{\lambda}

\newcommand{\si}{\sigma}

\newcommand{\Om}{\Omega}

\newcommand{\mto}{\mapsto}
\newcommand{\ra}{\rightarrow}

\newcommand{\lra}{\longrightarrow}
\newcommand{\ti}{\tilde}

\newcommand{\vzv}{\Leftrightarrow}

\newcommand{\lgl}{\ensuremath{\langle}}
\newcommand{\rgl}{\ensuremath{\rangle}}

\newcommand{\diam}{\ensuremath{diam}}

\newcommand{\ind}{\mathbf{1}}

\newcommand{\lqq}{\leqslant}
\newcommand{\gqq}{\geqslant}

\newtheorem{thms}{Theorem}[section]
\newtheorem{cors}[thms]{Corollary}
\newtheorem{lems}[thms]{Lemma}

\newtheorem{thm}{Theorem}[section]

\newtheorem{prp}[thm]{Proposition}

\newtheorem{lem}[thm]{Lemma}

\newtheorem{rem}[thm]{Remark}

\DeclareMathSymbol{\ophi}{\mathalpha}{letters}{"1E}

\renewcommand{\phi}{\varphi}

\newcommand{\be}{\begin{equation}}
\newcommand{\ee}{\end{equation}}
\newcommand{\ben}{\begin{equation*}}
\newcommand{\een}{\end{equation*}}

\newcommand{\ba}{\begin{equation}\begin{aligned}}
\newcommand{\ea}{\end{aligned}\end{equation}}

\DeclareMathOperator{\supp}{supp}

\DeclareMathOperator{\dist}{dist}

\newenvironment{proof}{\par\noindent{\bf Proof:}}{\hfill$\blacksquare$\par}

\newfont{\cyrfnt}{wncyr10}
\def\J3{\cyrfnt{\rm \u{\cyrfnt I}}}
\def\j3{\cyrfnt{\rm \u{\cyrfnt i}}}

\usepackage[]{color}
\definecolor{DarkGreen}{rgb}{0.1,0.7,0.3}   

\definecolor{DarkGreen}{rgb}{0.1,0.7,0.3}   

\allowdisplaybreaks[4]

\begin{document}
\title{
Strong averaging along 
foliated L\'evy diffusions 
with heavy tails \\
on compact leaves 
}

\author{
Michael A. H\"ogele
\footnote{Departamento de Matem\'aticas, Universidad de los Andes, Bogot\'a, Colombia; ma.hoegele@uniandes.edu.co}
\hspace{2cm}
Paulo Henrique da Costa
\footnote{Departamento de Matem\'atica, Universidade de Bras\'ilia, Bras\'ilia, Brazil; phcosta@unb.br} 
}

 \maketitle

\begin{abstract}

This article shows a strong averaging 
principle for diffusions driven by discontinuous heavy-tailed L\'evy noise, 
which are invariant on the compact horizontal leaves of a foliated manifold 
subject to small transversal random perturbations. 
We extend a result for such diffusions with exponential moments 
and bounded, deterministic perturbations 
to diffusions with polynomial moments of order $p\gqq 2$, 
perturbed by deterministic and stochastic integrals with unbounded coefficients and polynomial moments. 
The main argument relies on a result 
of the dynamical system for each individual jump increments 
of the corresponding canonical Marcus equation. 
The example of L\'evy rotations on the unit circle 
subject to perturbations by a planar L\'evy-Ornstein-Uhlenbeck process 
is carried out in detail. 
\end{abstract}



\section{Introduction}

The theory of averaging (deterministic) ordinary differential equations, whose origins date 
back to the works of Laplace and Lagrange,  
has  been  applied through its history in many fields of applications 
such as celestial mechanics, nonlinear mechanics, oscillation theory and radiophysics. 
First rigorous results start with the foundational contributions of Krylov, Bogoliubov and Mitropolskii 
\cite{An60, BK37, BM61, KK55, VM68}. 
For a comprehensive and systematic introduction to the subject we refer to the monograph \cite{SVM}
of Saunders, Verhulst and Murdock. 
The idea of averaging random systems given as stochastic ordinary differential equations 
with respect to Gaussian processes goes back to Stratonovich \cite{St63, St67}, 
first rigorous results appear with the seminal works of Khasminski and others \cite{Kh631, Kh632, Kh66, Kh68, GS72, FW84, Sk87}. 
More recent developments on systems of stochastic (partial) 
differential equations with \textit{continuous} Gaussian noises
can be found for instance in \cite{Ar01, Ce09, KK01, PV00, RW12, Ve91, Ve99} 
and the references therein. 

However, in many contexts the Gaussian paradigm is known to be too limited. 
First results on averaging differential equations with respect to 
\textit{discontinuous} and non-Gaussian Poisson noise are obtained in \cite{BS74} and \cite{KM91}. 
A first strong averaging principle for scalar L\'evy diffusions with 
Lipschitz coefficients and bounded jumps is established in \cite{JDX11}. 
In \cite{HR} the authors show a strong averaging principle 
for L\'evy diffusions with exponential moments on foliated manifolds, 
explained in more detail below. 	
This article is dedicated to the averaging of a large 
class of discontinuous semimartingales supported on the compact 
leaves of a foliated manifold having polynomial moments.  

Intuitively speaking a foliated manifold is a Riemannian manifold equipped with a 
family of equivalence classes of submanifolds depending on a parameter, 
which defines a ``transversal'' component. Precise details on foliated manifolds
are found in classical texts on the topic such as \cite{Cannas, Tondeur, Walcak}. 
The notion of a foliated Brownian motion is introduced in the seminal article \cite{Ga83} by Garnett. 
In \cite{Li} Li shows an averaging principle for Hamiltonian systems, which 
inspired first results on averaging on foliated spaces in \cite{GR13}, 
where the authors show an averaging principle for a foliated Brownian diffusion. 
More precisely, the authors study a Brownian diffusion in Stratonovich sense on 
a foliated manifold, which respects the foliated structure of the manifold in the sense that 
the diffusion does not leave the compact leaf of its initial condition almost surely. 
By assumption, the diffusion enjoys a unique invariant measure 
supported on the leaf of its initial condition as the limit of its time average in $L^p$ sense for some $p\gqq 2$. 
For small $\e$ the law of the perturbed 
diffusion converges on the accelerated time scale $t/\e$ to the invariant measure, 
such that the transversal perturbation converges to the vector field 
averaged against the invariant measure, which leads to an 
ordinary differential equation in transversal direction to the leaves of the foliation. 
Their main result is an averaging principle in 
the $L^p$ sense with logarithmic rates of convergence. 

This result is extended in \cite{HR} to a class of foliated 
L\'evy diffusions with compact leaves and exponential moments of the underlying jump L\'evy process, 
formulated in terms of canonical Marcus equations, as for instance in \cite{KPP95}. 
However by the L\'evy-Chinchine formula (see for instance \cite{Ap09}) it is obvious that this 
is a rather narrow subclass of possible L\'evy drivers. 
The current article generalizes this result to the case of L\'evy jump diffusions, 
with moments of order $p$, $p\gqq 2$, for which the averaging converges in $L^p$ sense. 
The lower bound $p\gqq 2$ seems natural for equations on manifolds 
since the Marcus canonical integral can be rephrased as an integral against the quadratic variation 
of the underlying process, see \cite{KPP95}, Lemma 2.1. 

The difficulty of an immediate extension of the results to L\'evy diffusions 
with only $p$-th moments lies in their formulation as a canonical Marcus equation \cite{KPP95, Ma78, Ma81}, 
where each single jump increment of the process is given as the solution of an ode, 
with a vector field tangential to the leaf. 
Assume that $\Delta Z$ is a single jump increment in an appropriate noise space 
of the driving L\'evy process $Z$. 
The jump increment on the manifold then has to follow the local coordinates 
and is then given as the increment 
$\Phi^{F\Delta Z}(y) -y$, where  
$\Phi^{F\Delta Z}(y)= Y(1; y, F\Delta Z)$ is the time $1$ map of the 
solution $Y$ of the ordinary differential equation 
\begin{equation}\label{eq: ode}
\frac{dY}{dt} = F(Y) \Delta Z, \qquad Y(0) = y,  
\end{equation}
and $F$ is a Lipschitz vector field such that $F\Delta Z$ 
takes values in the tangent space of our manifold. 
For details we refer to \cite{Ma78, Ma81, KPP95}. 
The problem is that in general the Lipschitz continuity of $F$ only implies 
\begin{equation}\label{eq: exponential increment}
|\Phi^{F\Delta Z}(x) - \Phi^{F\Delta Z}(y)| \lqq e^{\ell \|\Delta Z\|}|x-y|\qquad \forall x, y, 
\end{equation}
where $\ell$ is the Lipschitz constant of $F$.  
This means each jump increment on the manifold 
depends exponentially on the random size of $\Delta Z$. See \cite{KPP95} Lemma 3.1. 
Taking the expectation of (\ref{eq: exponential increment}) the finiteness of the right-hand side
implies the exponential integrability $\int_{\|z\| >1} \exp(\kappa \|z\|) \nu(dz) < \infty$ of the L\'evy measure $\nu$ of $Z$ 
for some constant $\kappa>0$ larger than the Lipschitz constant $\ell$ of $F$,  
which is equivalent to the existence of exponential moments of $Z$. 
This straight-forward argument is the main reasoning 
concerning the moments carried out in the previous article \cite{HR}.
However, since the leaves of the foliation are compact and 
the main driving diffusion $X$ is invariant on the leaf of its initial condition, 
any jump increment of $X$ is bounded by the diameter of the leaf in the surrounding space.  
With this intuition in mind we may prove in Lemma \ref{lem: invariance lemma} of Section 3  
the following result on positive invariant ODE dynamical systems $Y$ 
of type (\ref{eq: ode}), which yields a global constant $C>0$ such that for any $x, y$
\begin{align}\label{eq: polynomial increment}
\sup_{t\gqq 0} 
|(DF(Y(t;x)\Delta Z) F(Y(t;x))\Delta Z - (DF(Y(t;y)\Delta Z) F(Y(t;y))\Delta Z| \lqq C|x-y| \,\|\Delta Z\|^2. 
\end{align}
Taking the expectation of (\ref{eq: polynomial increment}) a finite right-hand side is equivalent to 
$\int_{\|z\|> 1} \|z\|^2 \nu(dz)<\infty$ imposing only second moments of $Z$. 
It turns out eventually to be an easy task to link (\ref{eq: exponential increment}) 
to (\ref{eq: polynomial increment}) via Taylor expansion of $\Phi^{F\Delta Z}$.  
We follow these lines of reasoning in a technically more subtle setting 
in Section~\ref{sec: transversal perturbations}. 

The second extension we undertake is the step from 
perturbations by a small deterministic bounded vector field in \cite{HR}
to a general class of discontinuous L\'evy diffusions with moments of order $2p$, 
whose multiplicative coefficients may depend on the slow component. 
The coefficients in front of $\circ d B$ and $\diamond \ti Z$ 
will be only depend on $\pi X^\e$, since it is well-known in averaging theory 
that in general diffusion coefficients are difficult to average in a strong sense. 

The article is organized as follows. 
Subsection~2.1 lays out the general setup. 
Subsection~2.2 states the specific hypotheses on 
the integrability and ergodicity conditions of the stochastic processes and 
the main result of this article given in Theorem \ref{thms: main result 1}. 
Subsection~2.3 spells out the main example:  
L\'evy processes with polynomial moments on the unit circle. 
In Section~3 we establish the estimate (\ref{eq: polynomial increment}) and 
derive the crucial estimates on the deviation of the perturbed from the unperturbed solution, 
under arbitrary Lipschitz functions including a crucial dynamical system argument. 
Section~4 is dedicated to the control of the averaging error term 
exploiting the results from Section 3 in special cases. 
Section~5 finishes the proof of the main result synthesizing Section~3 and 4. 
The article finishes with an Appendix providing the missing details of the example.

\section{Object of study and main results} 
\subsection{The set up } \label{subsec: setup}

\paragraph{The geometry: } 
Let $M$ be a finite dimensional connected, smooth Riemannian manifold. 
It is known by the strong version of Whitney theorem for instance in Boothby \cite{Boothby} 
that any finite dimensional smooth manifold is embedded in $\mathbb{R}^{m}$ for some $m \in \NN$ sufficiently large.  
The manifold $M$ is equipped with an $n$-dimensional foliation $\fM$ in the following sense. 
Let $\fM = (L_{x})_{x\in M}$, with $M = \bigcup_{x\in M} L_x$ and the sets $L_x$ are 
equivalence classes of the elements of $M$ satisfying the following properties. 
\begin{enumerate}
 
 \item[a)] 
Given an $x_0\in M$, there exists a neighborhood 
$U\subset M$ of the corresponding leaf $L_{x_0}$, 
a connected open set $V\subset \RR^d$ containing the origin $0\in \RR^d$
and a diffeomorphic coordinate map 
$\varphi: U \rightarrow L_{x_0}\times V$. 
\item[b)] The set $U$ of item a) can be taken small enough such 
that the derivatives of the coordinate map $\varphi$ are bounded.  
The second coordinate of a point $x \in U$, called the 
vertical coordinate, is denoted with 
the help of the projection $\pi: U \rightarrow V$ by 
$\varphi(x)=(\bar x, \pi(x))$ for some $\bar x\in L_{x}$.  
\end{enumerate}

\begin{rem}
For any $v\in V$ and $x \in U$ with $\pi(x)=v$ the preimage satisfies $\pi^{-1}(v)= L_x$. 
\end{rem}

\paragraph{The unperturbed equation: } 
We are interested in the ergodic behavior of a 
strong solution of a L\'evy driven SDE with discontinuous components which takes 
values in $M$ and respects the foliation. 
In order to avoid that jump increments lead to an exit from the foliation of the initial condition, 
the jump increments must respect the curved structure of the local coordinates and 
therefore necessarily satisfy a canonical Marcus equation, 
which are equivalent to the generalized Stratonovich equation in the sense of Kurtz, Pardoux and Protter \cite{KPP95}.  
We consider the formal canonical Marcus stochastic differential equation 
\begin{equation}\label{eq: SDE}
d X_t = F_0(X_t) dt + F(X_t) \diamond d Z_t + G(X_t)\circ dB_t, \qquad X_0 = x_0\in M,
\end{equation}
which consists of the following components. 

\begin{enumerate}
\item[1.] Let $Z = (Z_t)_{t\gqq 0}$ with $Z_t = (Z^1_t, \dots, Z^r_t)$ be a L\'evy process with values in 
$\RR^r$ for fixed $r\in \NN$ on a given filtered probability space $\mathbf{\Omega} = (\Omega, \fF, (\fF_t)_{t\gqq 0}, \PP)$ 
with characteristic triplet $(0, \nu, 0)$. 
Suppose that the filtration $(\fF_t)_{t\gqq 0}$ satisfies the ``usual'' conditions in the sense of Protter \cite{Pr04}.
As a consequence of the L\'evy-It\^o decomposition 
$Z$ is a pure jump process with respect to a $\si$-finite measure $\nu: \bB(\RR^r) \ra [0, \infty]$ called the L\'evy measures  
satisfying 
\begin{equation}\label{eq: second order moment}
\int_{\RR^r} (1\wedge \|z\|^2)\; \nu(dz)< \infty\qquad \mbox{ and }\quad \nu(\{0\}) = 0.
\end{equation}
For details we refer to the monographs of Sato \cite{Sa99} or Applebaum \cite{Ap09}.

\item[2.] Let $F_0 \in \cC^2(M, T\fM)$ with $F_0(x) \in T_x L_x$. 
The vector field $F \in \cC^2(M; L(\RR^{r}; T\fM))$ satisfies 
that the map $M \ni x\mapsto F(x)$ is $\cC^2$ 
and the linear map $F(x)$ sends a vector $z\in \RR^r \mto F(x) z \in T_x L_x$ 
to the tangent space of the respective leaf. 
 \item[3.] Let $B = (B^1, \dots, B^r)$ be a standard Brownian motion with values in $\RR^r$ defined 
on $\mathbf{\Om}$ and $G\in \cC^2(M, L(\RR^r, T\fM))$ with $G(x) \in T_x L_x$ for any $x\in M$. 
\end{enumerate}
We further assume that the vector fields $F_0$, $F$, $(DF_0) F_0$, $(DF) F$, $G$ and $(DG) G$ are globally Lipschitz 
continuous with Lipschitz constant $\ell>0$. 

A strong solution of the formal equation (\ref{eq: SDE}) 
is defined as a map $X: [0, \infty) \times \Om \ra~M$ 
satisfying $\PP$-almost surely for all $t\gqq 0$ 
\begin{align}\label{eq: SDE unperturbed integral form}
X_t &= x_0 + \int_0^t F_0(X_s) ds + \int_0^t G(X_s) dB_s + \frac{1}{2}\int_0^t (DG(X_s)) G(X_s) d\lgl B\rgl_s \nonumber\\
&\quad+\int_0^t F(X_{s-}) d Z_s + \sum_{0 < s\lqq t} (\Phi^{F \Delta_s Z}(X_{s-})-X_{s-}- F(X_{s-}) \Delta_s Z),
\end{align}
where $\lgl B\rgl_\cdot$ stands for the quadratic variation process of $B$ in $\RR^r$ and 
the function $\Phi^{Fz}(x) = Y(1, x ; Fz)$ and $Y(t, x; Fz)$ for 
the solution of the ordinary differential equation 
\begin{equation}\label{eq: increment ode}
\frac{d}{d\si } Y(\si) = F(Y(\si)) z,  \qquad
Y(0) = x \in M, \quad z\in \RR^r. 
\end{equation}

\paragraph{The perturbed equation: } This article studies the situation where an SDE in the sense of 
(\ref{eq: SDE unperturbed integral form})  
which is invariant on the leaf of the initial condition $x_0$ is perturbed by a
transversal smooth vector field $\e K dt$ and the stochastic differentials $\e \ti G \circ d\ti B$ 
and $\e \ti K \diamond d\ti Z_t$ with $\e>0$ 
in the limit of $\e \searrow 0$. More precisely, we denote by $X^\e$, $\e>0$ 
the solution in the sense of equation (\ref{eq: SDE unperturbed integral form}) 
of the formal perturbed system 
\begin{align}
d X^\e_t &= F_0(X^\e_t) dt +  F(X^\e_t) \diamond d Z_t + G(X^\e_t) \circ dB_t 
+ \e \Big(K(X^\e_t) dt + \ti K(\pi(X^\e_t)) \diamond d\ti Z_t + \ti G(\pi(X^\e_t)) \circ d\ti B_t\Big),\nonumber\\
X^\e_0 &= x_0,\label{eq: SDE perturbed}
\end{align} 
where the additional coefficients are defined as follows. 
\begin{enumerate}
 \item[4.] Let $K: M \ra TM$ be a smooth vector field.  
 \item[5.] Let $\ti Z = (\ti Z^1, \dots, \ti Z^r)$ be a pure jump L\'evy process with 
values in $\RR^r$ defined on $\mathbf{\Om}$ and L\'evy measure $\nu'$ satisfying 
\[
\int_{\RR^r} (1\wedge \|z\|^2) \nu'(dz) < \infty \qquad \mbox{ and }\quad\nu'(\{0\}) = 0  
\]
and $\ti K \in \cC^2(V, L(\RR^r, TM))$.
 \item[6.] Let $\ti B = (\ti B^1, \dots, \ti B^r)$ be an $\RR^r$-valued Brownian motion defined on $\mathbf{\Om}$ 
 and $\ti G \in \cC^2(V, L(\RR^r, TM))$. 
 \end{enumerate}
 We assume that the vector fields $K$, $(DK) K$, $\ti K$, $(D \ti K) \ti K$, $\ti G$ and $(D\ti G) \ti G$ are globally Lipschitz 
 continuous with Lipschitz constant $\ti \ell>0$. 

\begin{thms}[\cite{KPP95}, Theorem 3.2 and 5.1]
Under the preceding setup in particular items a), b) and \mbox{1.- 3.,} there is a unique semimartingale $X$ 
which is a strong global solution of (\ref{eq: SDE}) on $\mathbf{\Om}$
in the sense of equation (\ref{eq: SDE unperturbed integral form}). It has a c\`adl\`ag version and 
is a (strong) Markov process. 
\end{thms}
\begin{rem}
 It is obvious that under the preceding setup in particular items a), b) and \mbox{1.-6.,} 
there is also a unique strong solution $X^\e$ on $\mathbf{\Om}$ of 
equation (\ref{eq: SDE perturbed}) in the analogous sense of equation 
(\ref{eq: SDE unperturbed integral form}) and with the same properties, if $F_0$ is replaced by $F_0 +\e K$ and $F$ by 
$(F, \e \ti K)$, $G$ by $(G, \e \ti G)$, $B$ by $(B, \ti B)$ and $Z$ by $(Z, \ti Z)$ accordingly. 
\end{rem}
We state the crucial chain rule for the Marcus equation given in \cite{KPP95} (Proposition 4.2).  
\begin{prp}
Let $Z$ and $F_0, F$ satisfy items 1) and 2) 
and $X$ be the solution of (\ref{eq: SDE unperturbed integral form}) with initial condition $x_0$ 
with $G=0$. Then for any $\Psi\in \cC^2(\RR^d)$ we have $\PP$-a.s. for all $t\gqq 0$
\begin{align*} 
\Psi(X_t) = \Psi(x_0) + \int_0^t (D\Psi)(X_s) F_0(X_s) ds + \int_0^t (D\Psi)(X_{s-})F(X_{s-}) \diamond dZ_s. 
\end{align*}
\end{prp}

\noindent A direct consequence of the chain rule 
is the following support property given as Proposition 4.3 in \cite{KPP95}. 
Each jump increment of the noise $\Delta Z$ 
is mapped to an increment $\Phi^{F\Delta Z}$ 
of the solution $X$ of (\ref{eq: SDE}).  
The increment $\Phi^{F\Delta Z}$ 
follows the integral curve $Y$ 
in (\ref{eq: increment ode}) along the vector field $Fz$ 
which is tangent to the (smooth) manifold $L_{x_0}$ and hence $M$. 
A standard support theorem for ODEs applied in the proof of this result 
then yields that the solution after the jump once again is an element of $L_{x_0}$ and hence $M$. 
This support property is maintained for an additional 
Stratonovich component, the reasoning is standard. 
For the solution $X^\e, \e>0$ given in (\ref{eq: SDE perturbed}) 
this remains obviously true only for $M$.  
Under the aforementioned conditions these lines of thought 
lead to the following foliated structure of $X$: 
$x_0\in M$ implies $X_t(x_0) \in L_{x_0}$, $\PP$-a.s. for all $t\gqq 0$.
We shall call a solution of an SDE of the type (\ref{eq: SDE}) which admits a foliated solution 
a \textit{foliated L\'evy diffusion}.

In addition, we obtain that $x_0 \in M$ and $\e>0$ imply $X^\e_t(x_0) \in M$, $\PP$-a.s. for all $t\gqq 0$. 

\subsection{The main result} 

\begin{enumerate}
\item[\textbf{Hypothesis 1: Compactness and Integrability.} ]
\begin{enumerate}
 \item Any leaf $L_{x_0}\in \fM$, $x_0 \in M$, is compact and the map 
$x_0 \mapsto \diam L_{x_0}$ is Lipschitz continuous in the embedding space of~$M$. 
 \item There is a constant $p\gqq 2$ such that the L\'evy measures $\nu$ (of $Z$) and $\nu'$ (of $\ti Z$) satisfy 
\begin{align*}
&\int_{\RR^r} \|z\|^{p} \,\nu(dz) < \infty \qquad \mbox{ and } \qquad \int_{\RR^r} \|z\|^{2p} \,\nu'(dz) < \infty.
\end{align*}
\end{enumerate}

 \item[\textbf{Hypothesis 2: Existence of invariant measures on each leaf.} ] 
\begin{enumerate}
  \item The solution $X$ of (\ref{eq: SDE}) has for any initial condition $x_0\in M$ 
  a unique invariant measure $\mu_{x_0}$ with $\supp(\mu_{x_0}) =~L_{x_0}$.
 \item For $v_0 = \pi(x_0)$ being the vertical coordinate of some $x_0\in M$ 
 we define for $h: M \ra T\fM$ with $h(x) \in T_x M$  
\begin{equation}\label{def: average}
Q^h(v_0) := \int_{L_{x_0}} h(u) \mu_{x_0}(du) 
\end{equation}
and suppose that for any such function $h$, which is globally Lipschitz continuous 
the function 
$\RR^d \supset V \ni v \mapsto Q^{h}(v) \in \RR^d$
is globally Lipschitz continuous. 
\end{enumerate}
\end{enumerate}
\begin{rem}
Note that $\mu_{x_0}$ (just as $L_{x_0}$) only depends on the vertical component $\pi(x_0)$. 
\end{rem}

\noindent Hypothesis 2 ensures 
that for each  $x_0\in M$, $v_0 = \pi(x_0)\in V$ the stochastic differential equation  
\begin{equation}\label{def: w}
dw = Q^{\pi K} \left(w \right) dt + \ti K(w) \diamond d\ti Z_t + \ti G(w) \circ d\ti B_t, \qquad w(0) = v_0\in V
\end{equation}
has a unique strong solution $w = (w(t, v_0))_{t\in [0, T_\infty)}$ on $\mathbf{\Om}$, $T_\infty$ being the 
first exit time of $w$ from~$V$.

\begin{itemize}
 \item[\textbf{Hypothesis 3: Ergodicity in terms of $L^p$.}] 
Let Hypotheses 1 and 2 be satisfied for some $p\gqq 2$. 
We assume that there exists a bounded, continuous, decreasing function $\eta: [0,\infty) \ra [0,\infty)$ 
with $\eta(t) \searrow 0$ as $t\ra \infty$ such that for any $x_0 \in M$ 
\begin{equation} \label{def: function eta}
\left( \mathbb{E} \left| \frac{1}{t}\int_0^t \pi K (X_s(x_0) )\, ds - 
Q^{\pi K} (\pi (x_0)) \right|^p \right)^{\frac{1}{p}} \lqq \eta(t), \qquad \mbox{ for all }t\gqq 0.
\end{equation}
 \end{itemize}
\noindent{\it On the rate of convergence.} 
Results about rates of convergence go back to Pascal \cite{MP86} 
and Kolmogorov \cite{KO54}, see also 
\cite{Ar83} and references therein. 
Recent developments for L\'evy driven dissipative systems 
can be found in Kulik \cite{Ku09}, see also \cite{Ba08, Br14,Du09,Vi09}. 
In \cite{Ku09,Do95} for instance the authors develop 
generic methods to establish exponential convergence 
to an ergodic limit measure in terms of the total variation distance 
for the solution of an SDE driven by a pure jump Markov process. 
In this context the exponential rate of convergence 
in total variation implies that the rate of convergence in (\ref{def: function eta}) is of order $1/t^p$. 
In Subsection \ref{subsec:example} we provide the simple example 
of a L\'evy process on the unit circle for which we calculate the precise rate of convergence, 
which is of the same type. 
For Brownian diffusion processes H\"ormander's hypoellipticity condition 
ensures exponential rates of convergence in total variation, 
see for instance \cite{Be87,Bi81, Ha11} and references therein. 
On the other hand, there is no standard rate of convergence for general Markovian systems 
in the ergodic theorem, see for instance Krengel \cite{Krengel} or Kakutani and Petersen 
\cite{Kakutani-Petersen}. Therefore, it is natural to formulate the result in terms of 
the function $\eta$ following the approach in Freidlin and Wentzell \cite{FW84}. 

The main result of the article is proved in Section 5 and reads as follows. 

\begin{thms}\label{thms: main result 1}
Let Hypotheses 1, 2 and 3 being satisfied for some $p\gqq 2$. 
Then we have for any $\la \in (0,1)$ and $x_0 \in M$ positive 
constants $\e_0\in (0,1)$, $C>0$ and $c>0$ 
 such that for any $\e\in (0, \e_0]$ and $T \in [0, 1]$ 
\begin{align}
\left(\EE\left[\sup_{t\in [0, T]} |\pi\big(X^\e_{\frac{t}{\e}\wedge \tau^{\e}}(x_0)\big) 
- w(t)|^p\right]\right)^{\frac{1}{p}} 
\lqq C T \left[ \e^{\la} + \eta \left( c T|\ln \epsilon | \right) \right],\label{eq: result Cp}
\end{align}
where $X^\e$ is the solution (\ref{eq: SDE perturbed}) and $w$ the solution of (\ref{def: w}), 
and $\tau^\e = S^\e \wedge T_\infty$. 
$S^\e$ is the first exit time of $X^\e(x_0)$ from $U$ in a) 
and $T_\infty$ is the first exit time of $w$ from $V$.\\
\end{thms}

\begin{rem}
Since our main result focuses on L\'evy processes with only p-th moments, 
the coefficient $G$ can be set to $0$ since no additional 
difficulty to the proof of \cite{GR13} shows up. 
The coefficient $\ti G$ will be also dropped in the proof. 
Including it in the proofs of the sections \ref{sec: transversal perturbations} and 
\ref{sec: proof of the main result} is straight-forward. 
\end{rem}

\begin{rem}
In the proofs of Section 3, 4 and 5, 
it will turn out that under the preceding 
assumptions none of the constants depends 
on the precise shape of $V$ and hence $U$. 
Hence without loss of generality and for the 
sake of readability we may assume in the proofs that $V = \RR^d$. 
\end{rem}

\subsection{Example: Perturbed L\'evy rotations of the unit circle}\label{subsec:example}
We illustrate this phenomenon in $M= \RR^2 \setminus \{0\}$
with the 1-dimension horizontal circular foliation of $M$ where the leaf
passing through a point $x_0\in M$ is given by the (nondegenerate) horizontal 
circle
\[
L_{x_0} = \{ (
\|x_0\|
\cos \theta, \|x_0\| \sin 
\theta),\quad \theta \in [0,2 \pi) \}.
\]
Let the process $Z = (Z_t)_{t\gqq 0}$ be any pure jump L\'evy process with second moments. 
The L\'evy-It\^o decomposition of $Z$ yields almost surely for any $t\gqq 0$ 
\begin{equation}\label{eq: Levy-Ito}
Z_t = \int_0^t \int_{|z|\lqq 1} z \ti N(dsdz) + \int_0^t \int_{|z|> 1} z N(dsdz)
\end{equation}
where $N$ is the random Poisson measure with intensity measure $dt \otimes \nu$ and
$\ti N$ denotes its compensated counterpart. 
Consider the foliated linear SDE on
$M$ consisting of random rotations:
\begin{equation}\label{eq: SDE example}
dX_t = \Lambda X_t ~\diamond dZ_t, \qquad X_0 = x_0,\qquad \mbox{ with }\Lambda 
= \left( 
\begin{array}{ccc}
0 & -1 \\
1 & 0 
\end{array} 
\right). 
\end{equation}
Equation (\ref{eq: SDE example}) is defined as follows.  
Note that for some jump increment of $Z$, $z\in \RR$, $z\neq 0$, 
we have to consider the solution flow $\Phi$ of the equation 
\[
\frac{d}{d\si }Y(\si) = F(Y(\si)) z, \qquad Y(0) = (x,y), \qquad \mbox{ where }F(x,y) = \Lambda (x,y)^T,
\]
obtained by a simple calculation as
\[
\Phi^{F z}(x,y) = Y(1; (x,y)) 
= \left(\begin{array}{c}  x \cos(z) - y \sin(z) \\ x \sin(z) + y \cos(z) \end{array}\right),
\]
such that 
\begin{align*}
X_t &= x_0 + \int_0^t \Lambda X_{s-} z \ti N(dsdz) +  \sum_{0 < s\lqq t} (\Phi^{F \Delta_s Z}(X_{s-}) - X_{s-}- F(X_{s-}) \Delta_s Z).
\end{align*}
The chain rule of the Marcus integral, Proposition 4.2 in \cite{KPP95}, 
states for $\|(x,y)^T\|^2 := x^2 + y^2$  
\begin{align}\label{ex: non radial part}
d \|X_t\|^2 = -2 X_{t-} \Lambda X_{t-} \diamond d Z_t = 0.
\end{align}
In fact, $X$ can be equally defined as the projection of $Z$ on the unit circle. 
If we identify the plane where $X$ takes its values with the complex plane $\CC$ 
we obtain $X_ t= e^{i Z_t}$. 
By the L\'evy-Chinchine representation of the characteristic function of $Z$ we obtain for any $p\in \RR$
\begin{align*}
\EE[X_t^p ]&= \EE[e^{ipZ_t}] = \exp(t \Psi(p)),  \quad \mbox{ where }& \quad\Psi(p)= \int_{\RR^d} (e^{ipz}-1-izp\ind\{|z|\lqq 1\}) \nu(dz).
\end{align*}
The invariant measures $\mu_{x_0}$ in the leaves $L_{x_0}$ passing through  points $x_0\in M$
are therefore given by normalized Lebesgue measures in the circle $L_{x_0}$ centered in 
$0$ with radius $\|x_0\|$. 
We are interested in the effective behavior of a small
transversal perturbation of order $\epsilon$:
\[
 dX^{\epsilon}_t =  
\Lambda X^{\epsilon}_t \;\diamond dZ_t  +
 \epsilon K (X^{\epsilon}_t)\ dt + \e d\ti Z_t
\]
with initial condition $x_0=(1,0)$, where $\ti Z$ is a pure jump L\'evy process with L\'evy measure $\nu'$ 
satisfying 
\[
\int_{\|z\|>1}\|z\|^4 \nu'(dz) < \infty.  
\]
We shall consider two classes of perturbing vector fields $K$. 

\noindent (A) Constant perturbation $\e K=\e (K_1, K_2)\in \RR^2$. 
This example was carried out in \cite{HR} for the Gamma process on the unit sphere, with $\ti Z = 0$. 
The case of a general L\'evy process is virtually identical. 
The main result in this case reads as follows. 
For any L\'evy process $Z$ with $\EE[|Z_1|^p]<\infty$, $p\gqq 2$ and $\la\in (0,1)$ 
we obtain $\e_0\in (0,1)$ and $T_0>0$ such that for any $T\in [0, T_0]$ and $\e \in (0, \e_0]$ we have 
\[
\left[\mathbb{E}\left( \sup_{s\in [0, T]} 
\left| \pi_r (X^{\epsilon}_{\frac{s}{\epsilon}}) - 1 
\right|^p\right)\right]^{\frac{1}{p}}\lqq \e^\la T.
\]

\bigskip 

\noindent (B) General linear perturbation $\e K(x,y) =  \e A (x,y)^T = \e (a x + by, cx +dy)^T$ 
for a given matrix $A \in \RR^{2\otimes 2}$, which 
is obviously globally Lipschitz continuous and smooth. 
The radial component of the vector field $K$ is then given by 
\begin{align*}
\pi_r K(\theta, r) 
&= r \langle (a \sin(\theta) + b \cos(\theta), c \sin(\theta) + d\cos(\theta))^T, (\sin(\theta), \cos(\theta))^T\rangle\\
&= r \big(a \sin^2(\theta) + d \cos^2(\theta) + (b+c) \sin(\theta)\cos(\theta)\big),
\end{align*}
where $\theta$ is the angular coordinate of $(x, y)$ 
whose distance to the origin is $r$. 
Hence the average of this component with respect to the invariant uniform 
measure on the leaves (circles) is given by  
\[ 
Q^{\pi_r K}(\theta, r) = \frac{1}{2\pi} \int_0^{2 \pi} \Pi_r K(\theta, r) d\theta = \frac{a+d}{2}r  
\]
for leaves $L_{x_0}$ with radius~$r$. 
We verify the convergence (\ref{def: function eta}) of Hypothesis 2 
for the radial component and $p=2$. 
Let $\ti Z$ be a L\'evy process in $\RR^2$ with finite fourth moment. 
Elementary but lengthy calculations 
which can be found in Appendix \ref{subsec: appendix1} 
show that
\begin{align}
&\EE\Big[\Big|\frac{1}{t}\int_0^t \pi_r K(X_s) ds - Q^{\pi_r K}(x_0)\Big|^2\Big]^\frac{1}{2}
\stackrel{t\ra \infty}{\lra} 0,\label{eq: convergencia Student L2}
\end{align}
where the rate of convergence $\eta$ is of order $1/\sqrt{t}$ as $t\nearrow\infty$.

For an initial value $x_0 = (r_0 \cos(\theta_0), r_0\sin(\theta_0))$ 
the transversal system stated in Theorem \ref{thms: main result 1} is then 
$w(t)= r_0e^{\frac{a+d}{2}t} r_0$. 
Hence the result guarantees that the radial part 
$\pi_r \big( X^{\epsilon}_{\frac{t}{\epsilon}\wedge \tau^{\epsilon}} \big)$ 
on the accelerated time scale $\frac{t}{\epsilon}$ has 
a local behavior close to the exponential $e^{\frac{a+d}{2}t}$ in the sense that for any $\la\in (0,1)$ 
there are constants $C, c_\la>0$ and $\e_0\in (0,1)$ such that for any $T\in [0,1]$ and $\e\in (0, \e_0]$ 
we have 
\[
\left(\mathbb{E}\left[ \sup_{s\in [0, T]} 
\left| \pi_r  \big( X^{\epsilon}_{\frac{s}{\epsilon}\wedge \tau^{\epsilon}}(x_0) 
\big) - r_0e^{\frac{a+d}{2}s}\right|
^2\right]\right)^{\frac{1}{2}}
\lqq  C T\Big(\e^{\la}  +  (c |\ln \epsilon |)^{-\frac{1}{2}}\Big), 
\]
where $c_\la$ is given in Corollary \ref{cor: preliminary}. 
This averaging error tends to zero for fixed $T$ 
when $\epsilon\searrow 0$ and for fixed $\e$ if $T\searrow 0$.  

\section{The perturbation error} \label{sec: transversal perturbations}

In order to prove the main theorem we have to 
control the error $X^\e-X$ in terms of $L^p$. 
This result relies on the following elementary 
but in this context crucial lemma on dynamical systems, 
which yields on the right-hand side only quadratic dependence 
on the ``jump increment'' $z$. Due to its importance 
for this article we provide a sketch of proof. 

\begin{lems}\label{lem: invariance lemma}
For a globally Lipschitz continuous 
matrix-valued vector field $F \in \cC^2(\RR^{r+n}, L(\RR^{r}, \RR^{r+n}))$ 
and $z\in \RR^r$ denote by $(Y(t; x, Fz))_{t\gqq 0}$ the unique global strong solution of 
the ordinary differential equation 
\begin{equation}\label{eq: increment equation}
\frac{d Y}{dt} = F(Y)z\qquad Y(0, x, Fz) = x\in \RR^{r+n}.
\end{equation}
\begin{enumerate}
 \item[1)] Then there is a constant $C>0$ such that for any $z \in \RR^r$ and $x, y \in M$ with $Y(t; x) = Y(t;x, Fz)$ we have 
 \begin{align*}
\sup_{t\gqq 0} 
|(DF(Y(t;x))z)F(Y(t;x))z - (DF(Y(t;y))z)F(Y(t;y))z|
\lqq C~|x-y| ~\|z\|^2.
\end{align*}
 \item[2)] For any $x\in M$ we have $~\sup_{t\in [0,1]} \|DF(Y(t;x))F(Y(t;x))\| < \infty.$
\end{enumerate}
\end{lems}

\begin{proof}
We lighten notation and omit the parameter $Fz$ in $Y$ and write $Fz = F_z$. 
By the change of variables we have for any $x \in M$, $t\gqq 0$ that 
\begin{align*}
F_z(Y(t;x)) = F_z(x) + \int_0^t DF_z(Y(s;x)) F_z(Y(s;x)) ds. 
\end{align*}
Differentiating in $t$ yields 
\begin{align*}
\frac{d}{dt} F_z(Y(t;x)) = DF_z(Y(t;x)) F_z(Y(t;x)).
\end{align*}
Hence for any $x,y \in L$ the mean value theorem and equation (\ref{eq: increment equation}) yield
\begin{align*}
& DF_z(Y(t;x)) F_z(Y(t;x)) - DF_z(Y(t;y)) F_z(Y(t;y)) \\
&= \int_0^1 \frac{d}{dt} DF_z(Y(t;x+\si (y-x))) (y-x) d \si \\
&= \int_0^1 (D^2 F_z)(Y(t;x+\si (y-x))) ((y-x), F_z(Y(t;x+\si (y-x))) d\si.
\end{align*}
Since $(DF) F$ is Lipschitz continuous and 
$F\in \cC^2$ the operator $D((DF) F)$ is uniformly bounded. 
The chain rule $D((D F) F) = (D^2 F) F + (DF)(DF)$ yields 
\[
\|(D^2 F) F\|_\infty \lqq \|(D^2 F) F + (DF)(DF)\|_\infty + \|(DF)(DF)\|_\infty <\infty,
\]
where $\|DF\|$ is uniformly bounded since $F$ is globally Lipschitz continuous and $F\in \cC^1$. 
Therefore 
\begin{align*}
&|DF_z(Y(t;x)) F_z(Y(t;x)) - DF_z(Y(t;y)) F_z(Y(t;y))| \lqq \|(D^2F) F\| \|z\|^2 |x-y|
\lqq C |x-y|\|z\|^2.
\end{align*}
Since the right-hand side is independent of $t$ we take the supremum as claimed in statement~1). 
Statement~2) is a straight-forward consequence of the product rule.  
\end{proof}

\begin{prp}\label{lem: preliminary} 
Let the assumptions of Subsection 2.1 and Hypotheses 1, 2 and~3 be satisfied for some $p\gqq 2$. 
Then for any Lipschitz function $h: M \ra \RR$ 
there exist positive constants $\e_0, k_0, k_1, k_2$ with $k_0< 1$ 
such that for all $T\gqq 0$ satisfying $\e_0 T\lqq k_0$, 
$\e\in (0, \e_0]$ implies
\begin{align}\label{eq: compact bounded}
\left(\EE\left[\sup_{t\lqq T}|h(X^\e_t(x_0) ) - h(X_t(x_0))|^p\right]\right)^\frac{1}{p} 
\lqq k_1 \e^{1} \exp(k_2 T).
\end{align}
In addition, the constant $k_2$ is a polynomial in $\diam L_{x_0}$ of order $p$ with positive coefficients.  
\end{prp}

\begin{cors}\label{cor: preliminary}
Let the assumptions of Proposition \ref{lem: preliminary} be satisfied for some $p\gqq 2$. 
Then for any $\la \in (0,1)$ given there exist positive constants 
$c_\la$, $\e_0$, $k_3$ such that $T_\e := -c_\la \ln(\e)$, $\e\in (0, \e_0]$ satisfies  
\begin{equation}\label{eq: cor preliminary}
\left(\EE\left[\sup_{t\lqq T_\e}|h(X^\e_t(x_0) ) - h(X_t(x_0))|^p\right]\right)^{\frac{1}{p}} 
\lqq k_3 \e^\la. 
\end{equation}
In addition, the constant $k_3$ is a polynomial in $\diam L_{x_0}$ of order $p$ with positive coefficients. 
\end{cors}

\begin{proof}
Plugging $T_\e = - c\ln(\e)$ into the right-hand 
side of (\ref{eq: compact bounded}) we obtain $k_1 \e \exp(k_2 T_\e) = k_1 \e^{1-c k_2}.$
Given $\la \in (0,1)$ we choose $c = \frac{1}{k_2} \big(1 - \la'\big)$ and $\la' = \frac{1}{2}\big(\la +1\big)$ to infer the desired result.
\end{proof}

\bigskip
\begin{proof} (of Proposition \ref{lem: preliminary}) The proof consists 
in three parts. After changing the coordinates in part 1 we estimate 
the transversal component $|v^\e-v|$ using Lemma \ref{lem: invariance lemma} in part 2. 
In part 3 we estimate the horizontal component $|u^\e -u|$ before concluding 
with a nonlinear comparison principle. Part 2 and part 3 are given in separate lemmas. 
The main tools to derive two (nonlinear) comparison principles 
are Lemma  \ref{lem: invariance lemma} and Kunita's maximal inequality for the $L^p$ norm ($p \gqq 2$) 
of the supremum of compensated Poisson integrals found in \cite{Ku04} and 
an extension of this result for $p\in [1,2]$ by Saint Loubert Bi\'e \cite{SLB98}. 

\paragraph{I. Change of coordinates: }
First we rewrite the respective solutions of equation (\ref{eq: SDE}) and (\ref{eq: SDE perturbed}), 
$X$ and $X^\e$, in terms of the coordinates given by the diffeomorphism $\phi$ 
\begin{align*}
(u_t, v_t) := \phi(X_t) \qquad &\mbox{ and }\qquad (u^\e_t, v^\e_t) := 
\phi(X^\e_t), \qquad \e\in (0,1), t\in [0, T].
\end{align*}
The Lipschitz regularities of $h$ and $\phi$ yield a joint Lipschitz constant 
$C_0 := Lip(h \circ \phi^{-1})$ such that 
\begin{align}\label{eq: ungleichung 1}
|h(X^\e_t)- h(X_t)| &= |h \circ \phi^{-1}(u^\e_t, v^\e_t) -h \circ \phi^{-1}(u_t, v_t) | \nonumber \\
&\lqq C_0 |(u^\e_t-u_t, v^\e_t- v_t)|\lqq C_0 (|u^\e_t-u_t|+| v^\e_t- v_t|).
\end{align}
The proof of the statement consists in calculating estimates for each summand on the right hand 
side of equation above. We define 
\begin{align*}
\ffF := (D\phi) \circ F\circ \phi^{-1},\qquad 
\fK := (D\phi) \circ K \circ \phi^{-1}, \qquad 
\ti \fK := (D\phi) \circ \ti K \circ \phi^{-1}, 
\end{align*}
whose derivatives are uniformly bounded. 
Considering the components in the image of $\phi$ we have: 
\[
\fK = (\fK_H, \fK_V), \qquad \ti \fK = (\ti \fK_H, \ti \fK_V)
\]
with $ \fK_H, \ti \fK_H \in TL_{x_0}$ with $\fK_H \perp \fK_V$ and 
$\fK_V, \ti \fK_V \in TV \simeq \RR^d$ with $\ti \fK_H \perp \ti \fK_V$. 
The chain rule for canonical Marcus equations (Theorem 4.2  of \cite{KPP95}) 
yields for equation (\ref{eq: SDE perturbed}) the following form of the components 
in $\phi$ coordinates  
\begin{align}
d u_t^{\e} &=  \ffF_0(u_t^{\e}, v_t^{\e}) dt +  \ffF(u_t^\e, v_t^\e) \diamond d Z_t 
+ \e \fK_H(u_t^\e, v_t^\e) dt + \e \ti \fK_H(v_t^\e) \diamond d \ti Z_t&\mbox{ with } 
u_t^{\e}\in L_{x_0}, 
\label{eq: foliations SDE in koordinaten}\\
d v_t^{\e} &= \e \fK_V(u_t^\e, v_t^\e) dt + \e \ti \fK_V(v_t^\e) \diamond d \ti Z_t&\mbox{ with } v_t^{\e}\in V.
\label{eq: transversale SDE in koordinaten}
\end{align}
Note that for $\e =0$ the equation yields $v_t = v^0_t = 0 \in V$ almost surely. 
However we will write $v_t$ nevertheless for the sake of readability.  

\begin{lems}[II. Estimate of the transversal deviation $|v^\e_\cdot-v_\cdot|$]\label{lem: vertical}
Under the previous assumptions we obtain the following. 
There is a constant $\bar C_1>0$ such that for $\e_0 T< 1$ 
$\e\in (0, \e_0]$ implies 
\begin{align}
\EE[\sup_{t\in [0, T]} |v^\e_t - v_t|^p] 
&\lqq \bar C_{1} \e^p (1+T^{2p+1}).\label{eq: vertical lem}
\end{align}
In addition $\bar C_1 = \bar C_1(\diam L_{x_0})$ depends globally Lipschitz continuously on $\diam L_{x_0}$. 
\end{lems}

\begin{proof} 
By assumption $F$, $K$ and $\ti K$ are globally Lipschitz continuous. 
Without loss of generality, we can assume they have all a common Lipschitz constant $\ell$. 
The compactness of $L_{x_0}$ and the embedding in $\RR^{n+d}$ yields the existence of $C_1 = \diam L_{x_0}$ such 
that $|u^\e_s(x_0)-u_s(x_0)|< C_1$.  
We use the notation 
\begin{equation}\label{eq: k1}
C_2(x_0) = \sup_{y\in L_{x_0}} \|\fK_V(y, 0)\|\lqq \|\fK(x_0)\| + \sup_{y\in L_{x_0}} \ell |x_0-y|\lqq \|\fK(x_0)\| +\ell\,\diam L_{x_0}< \infty,
\end{equation}
which is finite by the compactness of $L_{x_0}$ 
and the continuity of $\fK_V$. Keeping in mind that $\lgl D g(x), u\rangle = p |x|^{p-2}\langle x, u\rgl$ for $x \mapsto g(x) := |x|^p, x\in \RR^{n+d}$ 
we apply the change of variable formula and obtain 
\begin{align}
|v_t^\e - v_t|^p &= p \int_0^t |v^\e_s - v_s|^{p-2} |\lgl v^\e_s - v_s, \e\fK_V(u^\e_s, v^\e_s) \rgl| ds \nonumber\\
&\qquad + p  \int_0^t |v_{s-}^\e - v_{s-}|^{p-2} |\lgl v_{s-}^\e - v_{s-}, \e\ti \fK_V(v^\e_{s-}) \diamond d \ti Z_s\rgl| \nonumber \\
&\lqq p  \int_0^t |v^\e_s - v_s|^{p-1} |\e\fK_V(u^\e_s, v^\e_s) -\e\fK_V(u_s, v_s)| ds \tag{$H_1$}\\
&\quad + p \int_0^t |v^\e_s - v_s|^{p-1} |\e\fK_V(u_s, v_s)|ds \tag{$H_2$}\\
&\quad + p \int_0^t |v_{s-}^\e - v_{s-}|^{p-2} |\lgl v_{s-}^\e - v_{s-}, \e(\ti \fK_V(v^\e_{s-})-\ti \fK_V(v_{s-})) d\ti Z_s\rgl| \tag{$H_3$} \\
&\quad + p \int_0^t |v_{s-}^\e - v_{s-}|^{p-2} |\lgl v_{s-}^\e - v_{s-}, \e\ti \fK_V(v_{s-}) d\ti Z_s\rgl| \tag{$H_4$} \\
&\quad + p \sum_{0< s \lqq t} |v_{s-}^\e - v_{s-}|^{p-1} |\Phi^{\e \ti \fK_V \Delta_s \ti Z}(v^\e_{s-})-\Phi^{\e \ti \fK_V \Delta_s \ti Z}(v_{s-})\nonumber\\
&\qquad \qquad - (v^\e_{s-}-v_{s-}) - \e (\ti \fK_V(v^\e_{s-})- \ti \fK_V(v_{s-}))\Delta_s \ti Z| \tag{$H_5$}\\[2mm]
&\quad + p \sum_{0< s \lqq t} |v_{s-}^\e - v_{s-}|^{p-1} |\Phi^{\e \ti \fK_V \Delta_s \ti Z}(v_{s-})- v_{s-} - \e \ti \fK_V(v_{s-}) \Delta_s \ti Z| \tag{$H_6$}\\
&= H_1 + H_2 + H_3 + H_4 + H_5+H_6.
\end{align}
\paragraph{1. Pathwise representation and estimates: } 
$\mathbf{H_1}$: The compactness of $L_{x_0}$ yields for $C_3 = p \ell (1 + \diam L_{x_0})$  
\begin{align}
H_1 &\lqq \e p\ell \int_0^t |v^\e_s - v_s|^{p-1} \big(|u^\e_s-u_s| + |v^\e_s-v_s|\big) ds 
\lqq \e C_3 \int_0^t |v^\e_s - v_s|^{p} ds+ \e C_3 \int_0^t |v^\e_s - v_s|^{p-1} ds. \label{H1}
\end{align}
$\mathbf{H_2}$: A direct computation gives 
\begin{align}
H_2  &\lqq \e p C_2 \int_0^t |v^\e_s - v_s|^{p-1} ds\lqq \e C_4 \int_0^t |v^\e_s - v_s|^{p-1} ds. \label{H2}
\end{align}
$\mathbf{H_3}$: Switching to the Poisson random measure representation 
with respect to the compensated $\ti N'$, for instance in Kunita \cite{Ku04}, 
we obtain for $C_{5} = p \ti \ell \int_{\|z\|>1} \|z\| \nu'(dz)$
\begin{align}
H_3 
&= p\int_0^t \int_{\RR^r}|v_{s-}^\e - v_{s-}|^{p-2} \lgl v_{s-}^\e - v_{s-}, \e(\ti \fK_V(v^\e_{s-})-\ti \fK_V(v_{s-})) z \rgl \ti N'(ds dz)\nonumber\\
&\quad + p\int_0^t \int_{\|z\|>1} |v^\e_s - v_s|^{p-2} \lgl v^\e_s - v_s, \e(\ti \fK_V(v^\e_s)-\ti \fK_V(v_s)) z\rgl \nu'(dz) ds\nonumber\\
&\lqq \e p\int_0^t \int_{\RR^r}|v_{s-}^\e - v_{s-}|^{p-2} |\lgl v_{s-}^\e - v_{s-}, 
(\ti \fK_V(v^\e_{s-})-\ti \fK_V(v_{s-})) z \rgl| \ti N'(ds dz)\nonumber\\
&\quad + \e C_5 \int_0^t |v^\e_s - v_s|^{p}ds.\label{H3}
\end{align}
$\mathbf{H_4:}$ With the help of H\"older's inequality we obtain for $C_6 = p\ti \ell \int_{\|z\|>1} \|z\| \nu'(dz)$
\begin{align}
H_4 
&= p\int_0^t \int_{\RR^r}|v_{s-}^\e - v_{s-}|^{p-2} \lgl v_{s-}^\e - v_{s-}, \e\ti \fK_V(v_{s-}) z\rgl \ti N'(ds dz)\nonumber\\
&\quad + p\int_0^t \int_{\|z\|>1}|v^\e_s - v_s|^{p-2} \lgl v^\e_s - v_s, \e\ti \fK_V(v_s) z\rgl \nu'(dz) ds\nonumber\\
&\lqq \e p\int_0^t \int_{\RR^r}|v_{s-}^\e - v_{s-}|^{p-2} |\lgl v_{s-}^\e - v_{s-}, \ti \fK_V(v_{s-}) z\rgl| \ti N'(ds dz)\nonumber\\
&\quad + \e C_6 \int_0^t |v^\e_s - v_s|^{p-1} ds.\label{H4}
\end{align}
$\mathbf{H_5:}$ For the canonical Marcus terms, Lemma \ref{lem: invariance lemma}, statement 1), provides 
a positive constant $C_7$ which depends on the leaf of the initial condition such that 
independent of $\theta \in [0, 1]$
\begin{equation}\label{eq: second moments}
\Big|\Big[(D\ti \fK_V(\Phi^{\ti \fK_V z}(y,\theta)) z) \ti \fK_V\big(\Phi^{\ti \fK_V z}(y,\theta)\big)z\Big] 
- \Big[(D\ti \fK_V\big(\Phi^{\ti \fK_V z}(x,\theta)\big)z) \ti \fK_V\big(\Phi^{\ti \fK_V z}(x,\theta)\big)z\Big]\Big| 
\leqslant C_7 |x-y|\|z\|^2.
\end{equation}
The Poisson random measure representation 
of the random sum reads as the following estimate 
in terms of the quadratic variation of $Z$ for $C_8 = C_7 p /2$ 
\begin{align*}
H_5 
&\lqq \frac{\e^2 p}{2} \sum_{0< s\lqq t}  |v_{s-}^\e- v_{s-}|^{p-1} \Big|(D\ti \fK_V (\Phi^{\e \ti \fK_V \Delta_s \ti Z}(v^\e_{s-})) \ti \Delta_s Z)
\ti \fK_V \big(\Phi^{\e \ti \fK_V \Delta_s \ti Z}(v^\e_{s-})\big)\Delta_s \ti Z \nonumber\\
&\qquad - (D \ti \fK_V \big(\Phi^{\e \ti \fK_V \Delta_s \ti Z}(v_s)\big)\Delta_s \ti Z) 
\ti \fK_V\big(\Phi^{\e \ti \fK_V \Delta_s \ti Z}(v_s)\big)\Delta_s \ti Z \Big|\nonumber\\[2mm] 
&\lqq \e^2 C_8 \sum_{0< s\lqq t}  |v_{s-}^\e- v_{s-}|^{p} \|\Delta_s \ti Z\|^2.
\end{align*}
The representation of this sum 
in terms of the Poisson random measure, for instance in Kunita \cite{Ku04}, 
is given for $C_9 = \int_{\|z\|>1}\|z\|^2 \nu'(dz)$ by 
\begin{align}
\sum_{0< s\lqq t} |v^\e_{s-}-v_{s-}|^p \|\Delta_s \ti Z\|^2 
&= \int_0^t \int_{\RR^r} |v^\e_{s-}-v_{s-}|^p \|z\|^2 \ti N'(dsdz) 
+C_9 \int_0^t |v^\e_s-v_s|^p ~ds,\nonumber
\end{align}
which yields 
\begin{align}
H_5 &\lqq \e^2 C_{8} \int_0^t \int_{\RR^r} |v^\e_{s-}-v_{s-}|^p \|z\|^2 \ti N'(dsdz) 
+\e^2 C_{10}  \int_0^t |v^\e_s-v_s|^p ~ds.
\label{H5}
\end{align}
$\mathbf{H_6: }$ For the last term we use Lemma \ref{lem: invariance lemma}, statement~2), which yields a positive constant $C_{11}$ 
such that for any $z\in \RR^r$ 
\[
\sup_{\theta \in [0,1]} \|(D \ti \fK_V(Y(\theta, 0, \e \ti \fK_V z))z \ti \fK_V(Y(\theta ,0, \e \ti \fK_V z))z\| < C_{11} \|z\|^2<\infty.  
\]
Hence exploiting that $\int_{\|z\|>1} \|z\|^4 \nu'(dz) <\infty$ we have for $C_{12} = C_{11} (p/2) \int_{\|z\|>1} \|z\|^4 \nu'(dz)$
\begin{align}
H_6 &\lqq  p\sum_{0< s \lqq t} |v_{s-}^\e - v_{s-}|^{p-1} 
|\Phi^{\e \ti \fK_V \Delta_s \ti Z}(v_{s-})- v_{s-} - \e \ti \fK_V(v_{s-}) \Delta_s \ti Z|\nonumber\\
&\lqq \frac{\e^2 p}{2}\sup_{\substack{0\lqq s\lqq t\\ \theta\in [0,1]}} 
\|(D \ti \fK_V (\xi) \Delta_s \ti Z) \ti \fK_V (\xi) \Delta_s \ti Z )\|\bigg|_{\xi = Y(\theta,v_s, \e \ti \fK_V \Delta_s \ti Z)}  
\sum_{0< s \lqq t} |v_{s-}^\e - v_{s-}|^{p-1} \|\Delta_s \ti Z\|^2 \nonumber\\
&\lqq \e^2 C_{12} \sum_{0< s \lqq t} |v_{s-}^\e - v_{s-}|^{p-1} \|\Delta_s \ti Z\|^4\nonumber\\
&= \e^2 C_{12} \int_0^t \int_{\RR^r} |v_{s-}^\e - v_{s-}|^{p-1} \|z\|^4 \ti N'(dsdz)
 + \e^2 C_{13} \int_0^t  |v^\e_s - v_s|^{p-1}  ds.\label{H6}
\end{align}
Combining the estimates (\ref{H1}, \ref{H2}, \ref{H3}, \ref{H4}, \ref{H5}, \ref{H6}) we obtain 
\begin{align}
|v_t^\e - v_t|^p
&\lqq \e C_3 \int_0^t |v^\e_s - v_s|^{p}ds  + \e C_3 \int_0^t |v^\e_s - v_s|^{p-1}ds 
+ \e C_4 \int_0^t |v^\e_s - v_s|^{p-1} ds\label{eq: Ito v}\\
&\quad + \e C_5 \int_0^t |v^\e_s - v_s|^{p}ds 
+\e C_6 \int_0^t |v^\e_s - v_s|^{p-1} ds\nonumber\\
&\quad + \e^2 C_{10} \int_0^t |v^\e_s-v_s|^p  ~ds + \e^2 C_{13} \int_0^t  |v^\e_s - v_s|^{p-1}  ds\nonumber\\
&\quad + \e p \int_0^t \int_{\RR^r}|v_{s-}^\e - v_{s-}|^{p-2} |\lgl v_{s-}^\e - v_{s-}, 
(\ti \fK_V(v^\e_{s-})-\ti \fK_V(v_{s-})) z \rgl| \ti N^\prime(ds dz)\label{RPMI 3}\\
&\quad +\e p \int_0^t \int_{\RR^r}|v_{s-}^\e - v_{s-}|^{p-2} |\lgl v_{s-}^\e - v_{s-}, \ti \fK_V(v_{s-}) z\rgl| \ti N^\prime(ds dz)\label{RPMI 2}\\
&\quad + \e^2 C_8 \int_0^t \int_{\RR^r} |v^\e_{s-}-v_{s-}|^p \|z\|^2 \ti N^\prime(dsdz) \label{RPMI 4}\\
&\quad + \e^2 C_{12} \int_0^t \int_{\RR^r} |v_{s-}^\e - v_{s-}|^{p-1} \|z\|^4 \ti N'(dsdz).\label{RPMI 1}
\end{align}
\paragraph{2. Estimate of the marginal expectation: } 
Taking the expectation we obtain a constant $C_{14}$ such that 
\begin{align*}
\EE[|v^\e_t - v_t|^p] &\lqq \e C_{14} \int_0^t \big(\EE[|v^\e_s - v_s|^p] +  \EE[|v^\e_s - v_s|^{p-1}]\big) ds 
\end{align*}
\vspace{-3mm}
and
\begin{align*}
\sup_{s\in [0, t]} \EE[|v^\e_s - v_s|^p]
&\lqq \e C_{14} \int_0^t \big(\EE[|v^\e_s - v_s|^p] +  \EE[|v^\e_s - v_s|^{p}]^\frac{p-1}{p} \big) ds\\
&\lqq \e C_{14} t \sup_{s\in [0, t]} \EE[|v^\e_s - v_s|^p] +  C_{14} \e \int_0^t \EE[|v^\e_s - v_s|^{p}]^\frac{p-1}{p} ds.
\end{align*}
For $\e \in (0, \e_0]$ and $t\in [0, T]$, $T>0$ such that $\e_0 T C_{14} \lqq \frac{1}{2}$, 
that is fixing $k_0 = (2 C_{14})^{-1}$ in the statement, we have 
\begin{align*}
\sup_{s\in [0, T]} \EE[|v^\e_s - v_s|^p] &\lqq 2 C_{14} \e \int_0^T \sup_{t\in [0, s]} \EE[|v^\e_t - v_t|^{p}]^\frac{p-1}{p}ds.  
\end{align*}
It is easy to verify that 
the maximal solution of this equation is given by 
\begin{equation}
\sup_{s\in [0, T]} \EE[|v^\e_s - v_s|^p] \lqq \Big(\frac{2 C_{14}}{p} \e T\Big)^p\lqq C_{15} (\e T)^p. \label{eq: p}
\end{equation}
We can replace the exponent $p$ in the estimate of $\EE[|v_s^\e - v_s|^p]$ by $2p$ and $p-1$. 
This is possible since the integrals with respect to $\nu'$ do not depend on $p$ and obtain for $q \in \{2p, p-1\}$
\begin{align}
&\sup_{s\in [0, T]}\EE[|v^\e_s - v_s|^{q}] \lqq 
C_{16}(\e T)^{q}. \label{eq: 2p}
\end{align}

\paragraph{3. Estimate of the expectation of the supremum: } 
We go back to (\ref{eq: Ito v}) and note that all integrands 
with respect to the Lebesgue integral $ds$ are positive, 
such that the integrals are positive and increasing.  
Further we note that the last four summands are compensated Poisson random integrals $M = M^1 + M^2$. 
The sum of the terms (\ref{RPMI 3}, \ref{RPMI 2}, \ref{RPMI 4}) 
will be denoted for convenience by $M^1$ and term (\ref{RPMI 1}) will be denoted by $M^2$. 
We obtain with the help of Jensen's inequality 
\begin{align}
\EE[\sup_{[0, T]} |v^\e - v|^p] 
&\lqq \e  C_{17} \int_0^T \big(\EE[|v^\e_s - v_s|^p] +  \EE[|v^\e_s - v_s|^{p}]^\frac{p-1}{p}\big) ds + 
\EE[\sup_{[0, T]} |M^1|^2]^\frac{1}{2} + \EE[\sup_{[0, T]} |M^2|].\label{eq: final sup}
\end{align}
For $M^1$ we use Kunita's maximal inequality for the exponent $2$ (see \cite{Ku04} or \cite{Ap09}) and Young's inequality 
for the exponents $\frac{2p-2}{2p}$ and $2p$ in order to obtain 
\begin{align}
\EE[\sup_{[0, T]} |M^1|^2] 
&\lqq \e  C_{18} \Big(\|\ti \fK_V(0)\|\int_0^T \int_{\RR^r}\EE[|v^\e_s - v_s|^{2(p-1)}]  \|z\|^2 \nu'(dz) ds\nonumber\\
&\quad + \int_0^T \int_{\RR^r}\EE[|v^\e_s - v_s|^{2p}] \|z\|^2 \nu'(dz) ds\nonumber\\
&\quad + \int_0^T \int_{\RR^r} \EE[|v^\e_s - v_s|^{2 p}] \|z\|^4 \nu'(dz)ds\Big)\nonumber\\
&\lqq \e  C_{19} \int_0^T \Big(\EE\Big[|v^\e_s - v_s|^{2 (p-1)}\Big] + \EE\Big[|v^\e_s - v_s|^{2p}\Big]\Big) ds \nonumber\\
&\lqq C_{19} \int_0^T \Big((1+\e) \EE\Big[|v^\e_s - v_s|^{2p}\Big] + \e^{2p}\Big) ds, \label{eq: mart square}
\end{align}
and inserting (\ref{eq: 2p}) in (\ref{eq: mart square}) we get 
\begin{align}\label{eq: M12}
\EE[\sup_{[0, T]} |M^1|^2] 
&\lqq C_{20} \int_0^T (\e s)^{2 p} ds + \e^{2p} T \lqq C_{20} \e^{2p} T(1+T^{2p}).
\end{align}
This could be also repeated with term $M^2$ but the price to pay would be $\int \|z\|^8 \nu'(dz)<\infty$, 
equivalent to the finiteness eighth moments of $\ti Z$. 
Instead for $M^2$ we use a maximal inequality for integrals with respect to 
Poisson random measures for the exponent $1$ given in \cite{PZ08}, Lemma 8.22 and resp. Theorem 8.23, 
going back to Saint Loubert Bi\'e \cite{SLB98}, which states the existence of a constant 
$C_{21}$ such that for $C_{22} = C_{21} \int_{\RR^r} \|z\|^4 \nu'(dz)$ 
and $t\gqq 0$ 
\begin{align*}
\EE[\sup_{[0, T]} |M^2|] &\lqq 
\e^2 C_{22} \int_0^T \int_{\RR^r}  \EE\Big[|v^\e_s - v_s|^{p-1}\Big] \|z\|^4 \nu'(dz) ds 
\lqq \e^2 C_{22} \int_0^T  \EE\Big[|v^\e_s - v_s|^{p}\Big]^\frac{p-1}{p} ds.
\end{align*}
Inserting (\ref{eq: p}) in the preceding expression we get 
\begin{align}\label{eq: M2}
\EE[\sup_{[0, T]} |M^2|] &\lqq 
\e^2 C_{23} \int_0^T (\e s)^{p-1} ds \lqq C_{23} \e^{p+1} T^p.  
\end{align}
Transforming (\ref{eq: final sup}) with the help of (\ref{eq: p}), (\ref{eq: M12}) and (\ref{eq: M2}) 
and keeping in mind that $\e T <1$ yields 
\begin{align}
\EE[\sup_{[0, T]} |v^\e - v|^p] 
&\lqq \e C_{24} \int_0^T \big((\e s)^p  + (\e s)^{p-1}\big) ds + C_{25} \e^{p} \sqrt{T(1+T^{2p})} 
+ C_{23} \e^{p+1} T^p\nonumber\\
&\lqq C_{26} \e^p (1+T^{2p+1}).\label{eq: vertical}
\end{align}
\end{proof}

\begin{lems}[III. Estimate of the horizontal component $|u^\e_\cdot-u_\cdot|$]\label{lem: horizontal}
Under the previous assumptions we obtain the following. 
There are constants $\bar C_2, t_0>0$ such that for $\e_0 T< 1$ 
$\e\in (0, \e_0]$ implies 
\begin{align}
\EE\left[\sup_{t\in [0, T]}|u_t^{\e} - u_t|^p\right] 
&\lqq \bar C_{2} \e^{2p} (1+(T\vee t_0))^{2(p+1)} \exp\Big(\bar C_{2} T\Big).\label{eq: Gu estimate lem}
\end{align}
\end{lems}

\begin{proof}
For convenience of notation we restart with the numbering of constants. 
Formally we obtain 
\begin{align}
u^{\e}_t - u_t & =  \int_0^t (\ffF_0(u^{\e}_s, v^{\e}_s) - \ffF_0(u_s, v_s)) ds 
+ \int_0^t (\ffF(u_{s-}^\e, v_{s-}^\e) -\ffF(u_{s-}, v_{s-})) \diamond d Z_s \nonumber\\
& \qquad + \e\, \int_0^t \big(\fK_H(u^\e_s, v^\e_s) - \fK_H(u_s, v_s)\big) ds + \e \int_0^t \fK_H(u_s, v_s) ds 
+ \e \int_0^t \ti\fK_H(v_{s-}^\e) \diamond d \ti Z_s.
\label{eq: horizontal formal canonical Marcus}
\end{align}
This equation is defined in $\RR^{n}$ as 
\begin{align*}
u_t^{\e} - u_t & = \int_0^t [\ffF_0(u^\e_s, v^\e_s) - \ffF_0(u_s, v_s)] ds \\
&\quad + \int_0^t [\ffF(u_{s-}^\e, v^\e_{s-}) - \ffF(u_{s-}, v_{s-}) ]d Z_s\\
&\quad + \sum_{0 < s\lqq t} 
\big[(\Phi^{\ffF \Delta_s Z}(u^\e_{s-}, v^\e_{s-}) -\Phi^{\ffF \Delta_s Z}(u_{s-}, v_{s-}))\\
&\qquad\qquad -(u^{\e}_{s-} -u_{s-}, v^{\e}_{s-} -v_{s-}) - (\ffF(u^\e_{s-}, v^\e_{s-})-\ffF(u_{s-}, v_{s-})) \Delta_s Z\big] \\
&\quad+ \e\, \int_0^t \big(\fK_H(u^\e_s, v^\e_s) - \fK_H(u_s, v_s)\big) ds + \e \int_0^t \fK_H(u_s, v_s) ds 
+ \e \int_0^t \ti\fK_H(v_{s-}^\e) d\ti Z_s\\
&\quad + \sum_{0 < s\lqq t} 
\big[\Phi^{\e\ti\fK_H \Delta_s \ti Z}(u^\e_{s-}) - u^{\e}_{s-} - (\e\ti\fK_H(v^\e_{s-})) \Delta_s \ti Z\big].
\end{align*}
The change of variable formula for (\ref{eq: horizontal formal canonical Marcus}) yields formally 
\begin{align*}
 |u_t^{\e} - u_t|^p  
 & = p\int_0^t  |u^\e_s- u_s|^{p-2} \lgl u^\e_s- u_s, \ffF_0(u^\e_s, v^\e_s) - \ffF_0(u_s, v_s)\rgl ds \\
 & \quad + p\int_0^t  |u_{s-}^\e- u_{s-}|^{p-2} \lgl u_{s-}^\e- u_{s-}, (\ffF(u^\e_{s-}, v^\e_{s-}) - \ffF(u_{s-}, v_{s-})) \diamond dZ_s\rgl\\
 & \quad + \e\, p\int_0^t |u^\e_s- u_s|^{p-2} \lgl u^\e_s- u_s, \fK_H(u^\e_s, v^\e_s) - \fK_H(u_s, v_s)\rgl ds \\
 &\quad + \e p\int_0^t |u^\e_s- u_s|^{p-2} \lgl u^\e_s- u_s, \fK_H(u_s, v_s)\rgl ds\\
 & \quad + p\int_0^t  |u_{s-}^\e- u_{s-}|^{p-2} \lgl u_{s-}^\e- u_{s-}, \e \ti \fK_H(v^\e_{s-})\diamond d\ti Z_s\rgl.
\end{align*}
This is defined in $\RR^{n}$ as 
\begin{align}
 |u_t^{\e} - u_t|^p  
 & = p\int_0^t  |u^\e_s- u_s|^{p-2} \lgl u^\e_s- u_s, \ffF_0(u^\e_s, v^\e_s) - \ffF_0(u_s, v_s)\rgl ds \tag{$I_1$} \\
 & \quad + p\int_0^t  |u_{s-}^\e- u_{s-}|^{p-2} \lgl u_{s-}^\e- u_{s-}, (\ffF(u^\e_{s-}, v^\e_{s-}) - \ffF(u_{s-}, v_{s-})) dZ_s\rgl\tag{$I_2$} \\
& \quad + p\sum_{0< s\lqq t}  |u_{s-}^\e- u_{s-}|^{p-2} \lgl u_{s-}^\e- u_{s-}, 
\Phi^{\ffF \Delta_s Z}(u^\e_{s-}, v^\e_{s-}) -\Phi^{\ffF \Delta_s Z}(u_{s-}, v_{s-}) \nonumber\\
& \qquad \qquad \qquad \qquad  -(u^{\e}_{s-} -u_{s-}, v^\e_{s-}-v_{s-}) - (\ffF(u^\e_{s-}, v^\e_{s-})-\ffF(u_{s-}, v_{s-})) \Delta_s Z\rgl\tag{$I_3$}\\
 & \quad + \e\, p\int_0^t |u^\e_s- u_s|^{p-2} \lgl u^\e_s- u_s, \fK_H(u^\e_s, v^\e_s) - \fK_H(u_s, v_s)\rgl ds \tag{$I_4$}\\
 &\quad + \e p\int_0^t |u^\e_s- u_s|^{p-2} \lgl u^\e_s- u_s, \fK_H(u_s, v_s)\rgl ds \tag{$I_5$}\\
 &\quad + \e p \int_0^t |u_{s-}^\e- u_{s-}|^{p-2} \lgl u_{s-}^\e- u_{s-}, \ti \fK_H(v^\e_{s-}) d \ti Z_s \rgl  \tag{$I_6$}\\
  &\quad + p\sum_{0< s\lqq t} |u_{s-}^\e- u_{s-}|^{p-2} \lgl u_{s-}^\e- u_{s-}, \Phi^{\e \ti \fK_H\Delta_s \ti Z}(v_{s-}^\e)-\Phi^{\e \ti \fK_H\Delta_s \ti Z}(v_{s-})\nonumber\\
 &\qquad \qquad - (v^\e_{s-}-v_{s-})- \e (\ti \fK_H(v_{s-}^\e)-\ti \fK_H(v_{s-}))\Delta_s \ti Z\rgl \tag{$I_7$}\\[2mm]
  &\quad + p\sum_{0< s\lqq t} |u_{s-}^\e- u_{s-}|^{p-2} \lgl u_{s-}^\e- u_{s-}, \Phi^{\e \ti \fK_H\Delta_s \ti Z}(v_{s-})
- v_{s-}- \e \ti\fK_H(v_{s-})\Delta_s \ti Z\rgl \tag{$I_8$}\\
  &=: I_1 + I_2 + I_3 + I_4 + I_5 + I_6 + I_7 + I_8. \label{eq: pivot}
\end{align}

We now estimate the eight summands on the right-hand side one by one. 
The estimates $I_1$ and $I_4$ are straight forward Lipschitz estimates. 
For the stochastic It\^o terms we use the different kinds of maximal inequalities. 
The estimate of the canonical Marcus terms $I_3$, $I_7$ and $I_8$ 
is the most laborious task in which we exploit 
the result of Lemma \ref{lem: invariance lemma}.  
The term $I_5$ is straight forward. 

\paragraph{1. Estimate of the stochastic It\^o integral terms $I_2$ and $I_6$: } 
$\mathbf{I_2}$: Due to the existence of moments of order at least $1$, $I_2$ 
has the following representation with respect to the compensated Poisson random measure associated to $Z$
\begin{align}
&\int_0^t  |u_{s-}^\e- u_{s-}|^{p-2} \lgl u_{s-}^\e- u_{s-}, \big(\ffF(u^\e_{s-}, v^\e_{s-}) - \ffF(u_{s-}, v_{s-})\big) dZ_s\rgl\nonumber\\
&= \int_0^t \int_{\RR^r} 
|u_{s-}^\e- u_{s-}|^{p-2} \lgl u_{s-}^\e- u_{s-}, \big(\ffF(u^\e_{s-}, v^\e_{s-}) - \ffF(u_{s-}, v_{s-})\big) z \rgl \ti N(ds dz)\label{eq: term J1}\\
&\qquad + \int_0^t \int_{\|z\|>1} 
|u^\e_s- u_s|^{p-2} \lgl u^\e_s- u_s, \big(\ffF(u^\e_s, v^\e_s) - \ffF(u_s, v_s)\big) z \rgl \nu(dz) ds.\label{eq: term J2}
\end{align}
For the first term (\ref{eq: term J1}) we apply the embedding $L^2\subset L^1$, 
Kunita's maximal inequality (see \cite{Ku04} or \cite{Ap09}) for exponent equal to $2$, 
the compactness of $L_{x_0}$ implying $|u^\e_{\cdot}(x_0)-u_{\cdot}(x_0)|\lqq \diam L_{x_0}$, 
and the elementary Young inequality for the exponents $p$ and $p/(p-1)$ 
combined with inequality (\ref{eq: p}). We obtain a positive constant $C_1>0$ 
such that for $C_2 = 2 C_1 (\int_{\RR^r} \|z\|^2\nu(dz))^{1/2} \diam L_{x_0}^{p}$ and $C_3 = C_{23}$ 
we have 
\begin{align}
&\EE\left[\sup_{t\in [0, T]}
\Big|\int_0^t \int_{\RR^r} 
|u_{s-}^\e- u_{s-}|^{p-2} \lgl u_{s-}^\e- u_{s-}, \big(\ffF(u^\e_{s-}, v^\e_{s-}) - \ffF(u_{s-}, v_{s-})\big) z \rgl \ti N(ds dz)\Big| \right] \nonumber\\
&\lqq C_1 \EE\left[
\int_0^T \int_{\RR^r} 
|u^\e_s- u_s|^{2(p-2)} |\lgl u^\e_s- u_s, \big(\ffF(u^\e_s, v^\e_s) - \ffF(u_s, v_s)\big) z \rgl |^2 \nu(dz) ds\right]^\frac{1}{2} \nonumber\\
& \lqq C_1 
\EE\left[\int_0^T \int_{\RR^r} 
|u^\e_s- u_s|^{2(p-1)} \Big(|u^\e_s-u_s|^2 + |v^\e_s - v_s|^2\Big) \|z\|^2 \nu(dz) ds \right]^\frac{1}{2}\nonumber\\
&\lqq C_2\, \bigg(\int_0^T \EE\Big[\sup_{[0, s]} |u^\e- u|^{p} \Big]ds\bigg)^\frac{1}{2} 
+ C_{3} \e^p T^{p+\frac{1}{2}}
\label{I2 fertig}.
\end{align}
The second term is less delicate. Young's inequality for the exponents $p/(p-1)$ and $p$ 
and yields
\begin{align}
&\EE\Big[\sup_{t\in [0,T]} \int_0^t \int_{\|z\|>1} 
|u^\e_s- u_s|^{p-2} \lgl u^\e_s- u_s, \big(\ffF(u^\e_s, v^\e_s) - \ffF(u_s, v_s)\big) z \rgl \nu(dz) ds \Big]\nonumber\\
&\lqq \ell \int_{\|z\|>1} \|z\| \nu(dz) \EE\Big[\sup_{t\in [0,T]} \int_0^t 
\big(|u^\e_s- u_s|^{p} + |u^\e_s- u_s|^{p-1} |v^\e_s- v_s|\big) ds \Big]\nonumber\\
&\lqq \ell \int_{\|z\|>1} \|z\| \nu(dz)\Big( 2 \int_0^T ~\EE\Big[\sup_{[0,s]} |u^\e- u|^{p}\Big] ds 
+ \int_0^T  \EE\Big[|v^\e_s- v_s|^{p} \Big] ds\Big)\nonumber\\
&\lqq C_4 \int_0^T \EE\Big[\sup_{[0,s]} |u^\e- u|^{p}\Big] ds + C_5 \e^p  T^{p+1}\label{I2 comp}.
\end{align}
$\mathbf{I_6}$: Let $\ti N'$ be the compensated Poisson random measure associated to 
the L\'evy process $\ti Z$. Then 
\begin{align}
&\sup_{t\in [0, T]} \e \int_0^t |u_{s-}^\e- u_{s-}|^{p-2} \lgl u_{s-}^\e- u_{s-}, \ti \fK_H(v^\e_{s-}) d\ti Z_s\rgl\nonumber \\
&=\sup_{t\in [0, T]} \e \int_0^t \int_{\RR^r}  |u_{s-}^\e- u_{s-}|^{p-2} 
\lgl u_{s-}^\e- u_{s-}, (\ti \fK_H(v^\e_{s-})-\ti \fK_H(v_{s-}))z \rgl \ti N'(ds dz) \tag{$J_1$}\\
&\qquad +\sup_{t\in [0, T]} \e \int_0^t \int_{\|z\|>1} 
|u^\e_s- u_s|^{p-2} \lgl u^\e_s- u_s, (\ti \fK_H(v^\e_s)-\ti \fK_H(v_s)) z \rgl \nu'(dz) ds\tag{$J_2$}\\
&\qquad +\sup_{t\in [0, T]} \e \int_0^t \int_{\RR^r}  |u_{s-}^\e- u_{s-}|^{p-2} 
\lgl u_{s-}^\e- u_{s-}, \ti \fK_H(v_{s-})z \rgl \ti N'(ds dz) \tag{$J_3$}\\
&\qquad +\sup_{t\in [0, T]} \e \int_0^t \int_{\|z\|>1} 
|u^\e_s- u_s|^{p-2} \lgl u^\e_s- u_s, \ti \fK_H(v_s) z \rgl \nu'(dz) ds.\tag{$J_4$}
\end{align}
The terms $J_1$ and $J_2$ are structurally identical to (\ref{eq: term J1}) and (\ref{eq: term J2}) and 
are estimated analogously to (\ref{I2 fertig}) and (\ref{I2 comp}) where $\ffF$ is replaced by $\ti \fK_H$ 
which yield the estimates 
\begin{align*}
&\EE\Big[\sup_{t\in [0, T]} |\e \int_0^t \int_{\RR^r}  |u_{s-}^\e- u_{s-}|^{p-2} 
\lgl u_{s-}^\e- u_{s-}, (\ti \fK_H(v^\e_{s-})-\ti \fK_H(v_{s-}))z \rgl \ti N'(ds dz)|\Big]\\
&\lqq C_6 \left(\int_0^T \EE\Big[\sup_{[0,s]} |u^\e- u|^{p}\Big] ds\right)^\frac{1}{2} + C_7 \e^p T^{p+\frac{1}{2}} ~ 
~ \mbox{ and}\\
&\EE\Big[\sup_{t\in [0, T]} \e |\int_0^t \int_{\|z\|>1} 
|u^\e_s- u_s|^{p-2} \lgl u^\e_s- u_s, (\ti \fK_H(v^\e_s)-\ti \fK_H(v_s)) z \rgl \nu'(dz) ds| \Big]\\
&\lqq C_8 \int_0^T \EE\Big[\sup_{[0,s]} |u^\e- u|^{p}\Big] ds + C_9 \e^p  T^{p+1}.
\end{align*}
For the term $J_3$ we observe that $v_s = 0$ such that $\fK_V(v_s)$ is constant. 
Kunita's maximal inequality for the exponent $2$ yields the constant $C_{10}$ 
and the boundedness of $|u^\e_s-u_s|$ the constant $C_{11} = C_{10}\left(\int_{\RR^r} \|z\|^2 \nu'(dz) (\diam L_{x_0})^{p-2}\right)^\frac{1}{2}$ 
such that 
\begin{align*}
&\EE\Big[\sup_{t\in [0, T]} \e |\int_0^t \int_{\RR^r}  |u_{s-}^\e- u_{s-}|^{p-2} 
\lgl u_{s-}^\e- u_{s-}, \ti \fK_H(v_{s-})z \rgl \ti N'(ds dz)| \Big]\\
&\lqq \e \EE \Big[\sup_{t\in [0, T]}  |\int_0^t \int_{\RR^r}  |u_{s-}^\e- u_{s-}|^{p-2} 
\lgl u_{s-}^\e- u_{s-}, \ti \fK_H(v_{s-})z \rgl \ti N'(ds dz)|^2 \Big]^\frac{1}{2}\\
&\lqq \e C_{10} \left(\int_0^T \int_{\RR^r}  \EE\Big[ |u^\e_s- u_s|^{2(p-1)} \Big]  \|z\|^2 \nu'(dz) ds\right)^\frac{1}{2} \\
&\lqq \e C_{11} \left(\int_0^T \EE\Big[\sup_{[0,s]} |u^\e- u|^{p}\Big] ds\right)^\frac{1}{2}.
\end{align*}
The term $J_4$ is again easier, for $C_{12} = \int_{\|z\|>1} \|z\| \nu'(dz) \|\ti \fK_H(v_{s-})\|$ we obtain 
\begin{align*} 
&\EE[\sup_{t\in [0, T]} \e \int_0^t \int_{\|z\|>1} 
|u^\e_s- u_s|^{p-2} \lgl u^\e_s- u_s, \ti \fK_H(v_s) z \rgl \nu'(dz) ds ]\\
&~ \lqq \e C_{12}\int_0^T  \EE[\sup_{[0,s]} |u^\e- u|^{p-1}] ds\\
&~ \lqq C_{12}\int_0^T  \EE[\sup_{[0,s]} |u^\e- u|^{p}] ds + C_{12} \e^p T. 
\end{align*}
This yields 
\begin{align}
\EE[\sup_{[0, T]} |I_6|] 
&\lqq C_{13} \bigg(\bigg(\int_0^T \EE\Big[\sup_{t \in [0,s]} |u^\e- u|^{p}\Big] ds\bigg)^\frac{1}{2}
+ \int_0^T \EE\Big[\sup_{t \in [0,s]} |u^\e- u|^{p}\Big] ds 
+  \e^p T^p (2+ T) \bigg). \label{I6 fertig}
\end{align}

\paragraph{2. Estimate of the canonical Marcus terms $I_3$, $I_7$ and $I_8$: }
Lemma \ref{lem: invariance lemma} tells us that there is a positive constant, $C_{14}$, say, 
which depends on the leaf of the initial condition such that 
independent of $\theta \in [0, 1]$
\begin{equation}\label{Equ: second moments}
\Big|\Big[(D\ffF(\Phi^{\ffF z}(y,\theta)) z) \ffF\big(\Phi^{\ffF z}(y,\theta)\big)z\Big] \ 
- \Big[(D\ffF\big(\Phi^{\ffF z}(x,\theta)\big)z) \ffF\big(\Phi^{\ffF z}(x,\theta)\big)z\Big]\Big| 
\leqslant C_{14} |x-y|\|z\|^2.
\end{equation}
Now we apply the Taylor's theorem to $\theta \longmapsto Y(\theta;x, \ffF \Delta_s Z) = \Phi^{\ffF \Delta_s Z}(x,\theta)$ with $\theta \in [0,1]$ 
and inequality (\ref{Equ: second moments}). 
The Poisson random measure representation 
of the random sum writes as the following estimate 
in terms of the quadratic variation of $Z$ 
\begin{align}
|I_3| 
&\lqq \sum_{0< s\lqq t}  |u_{s-}^\e- u_{s-}|^{p-1} 
\big|\Phi^{\ffF \Delta_s Z}(u^\e_{s-}, v^\e_{s-}) -\Phi^{\ffF \Delta_s Z}(u_{s-}, v_{s-}) \nonumber\\
&\qquad -(u^{\e}_{s-} -u_{s-}, v^\e_{s-}-v_{s-}) - (\ffF(u^\e_{s-}, v^\e_{s-})-\ffF(u_{s-}, v_{s-})) \Delta_s Z\rgl\big|\nonumber\\ 
&\lqq \frac{1}{2}\sum_{0< s\lqq t}  |u_{s-}^\e- u_{s-}|^{p-1} \Big|(D\ffF(\Phi^{\ffF \Delta_s Z}(u^\e_{s-}, v^\e_{s-})) \Delta_s Z)
\ffF\big(\Phi^{\ffF \Delta_s Z}(u^\e_{s-}, v^\e_{s-})\big)\Delta_s Z \nonumber\\
&\qquad - (D\ffF\big(\Phi^{\ffF \Delta_s Z}(u_{s-}, v_{s-})\big)\Delta_s Z) \ffF\big(\Phi^{\ffF \Delta_s Z}(u_{s-}, v_{s-})\big)\Delta_s Z\Big|\nonumber\\[2mm] 
&\lqq C_{14} \Big(\sum_{0< s\lqq t}  |u_{s-}^\e- u_{s-}|^{p} \|\Delta_s Z\|^2 + \sum_{0< s\lqq t}  |u_{s-}^\e- u_{s-}|^{p-1} |v^\e_{s-} - v_{s-}| \|\Delta_s Z\|^2\Big)\nonumber\\
&\lqq 2 C_{14} \Big(\sum_{0< s\lqq t}  |u_{s-}^\e- u_{s-}|^{p} \|\Delta_s Z\|^2 
+ \sum_{0< s\lqq t}  |v^\e_{s-} - v_{s-}|^p \|\Delta_s Z\|^2\Big). \label{eq: Marcus estimate}
\end{align}
The representation of this sum 
in terms of the Poisson random measure, for instance in Kunita \cite{Ku04}, 
of the first term is 
\begin{align}
&\sum_{0< s\lqq t} |u^\e_{s-}-u_{s-}|^p \|\Delta_s Z\|^2 \nonumber\\
&= \int_0^t \int_{\RR^r} |u^\e_{s-}-u_{s-}|^p \|z\|^2 \ti N(dsdz) 
+\int_0^t \int_{\|z\|>1} |u^\e_s-u_s|^p  \|z\|^2 ~\nu(dz)~ds \label{eq: Marcus estimate 1st term u}
\end{align}
and the analogous result if $|u^\e_{s-}-u_{s-}|$ is replaced by $|v^\e_{s-}-v_{s-}|$. 
The maximal inequality for integrals with respect to the compensated Poisson random measures 
and inequality (\ref{eq: p}) yield  
\begin{align}
\EE[\sup_{[0, T]} |I_3|] 
&\lqq C_{15} \int_0^T \int_{\RR^r} \Big(\EE[\sup_{[0, s]} |u^\e-u|^p] 
+ \EE[|v^\e_s-v_s|^p]\Big) \|z\|^2 ~\nu(dz)~ds\nonumber\\
&= C_{15} \int_{\RR^r} \|z\|^2 \nu(dz) \Big(\int_0^T \Big(\EE[\sup_{[0, s]} |u^\e-u|^p] 
+ \EE[|v^\e_s-v_s|^p]\Big)  ~ds\Big)\nonumber\\
&\lqq C_{16} \Big(\int_0^T \EE[\sup_{[0, s]} |u^\e-u|^p]  ~ds
+ \e^p  T^p\Big).\label{I3 fertig}
\end{align}
$\mathbf{I_7}$: For $I_7$ we use Lemma \ref{lem: invariance lemma} statement 1) in terms of (\ref{eq: second moments}) 
and Young's inequality 
\begin{align*}
&\sum_{0< s\lqq t} |u_{s-}^\e- u_{s-}|^{p-2} \lgl u_{s-}^\e- u_{s-}, 
\Phi^{\e \ti \fK_H\Delta_s \ti Z}(v_{s-}^\e)-\Phi^{\e \ti \fK_H\Delta_s \ti Z}(v_{s-})\\
&\qquad - (v^\e_{s-}-v_{s-}) - \e (\ti \fK_H(v_{s-}^\e)-\ti \fK_H(v_{s-}))\Delta_s \ti Z\rgl \\[2mm]
&\lqq \sum_{0< s\lqq t} |u_{s-}^\e- u_{s-}|^{p-1} |\Phi^{\e \ti \fK_H\Delta_s \ti Z}(v_{s-}^\e)
-\Phi^{\e \ti \fK_H\Delta_s \ti Z}(v_{s-})
- (v^\e_{s-}-v_{s-}) - \e (\ti \fK_H(v_{s-}^\e)-\ti \fK_H(v_{s-}))\Delta_s \ti Z| \\
&\lqq \e^2 C_{17} \Big(\sum_{0< s\lqq t}  |u_{s-}^\e- u_{s-}|^{p-1} |v^\e_{s-} - v_{s-}| \|\Delta_s \ti Z\|^2\Big)\nonumber\\
&\lqq \e^2 C_{17} \Big(\sum_{0< s\lqq t}  \big(|u_{s-}^\e- u_{s-}|^{p} +  |v^\e_{s-} - v_{s-}|^p\big) \|\Delta_s \ti Z\|^2\Big).
\end{align*}
We rewrite the last expression in terms of the compensated Poisson random measure $\ti N'$ and obtain 
\begin{align}
&\sum_{0< s\lqq t}  \big(|u_{s-}^\e- u_{s-}|^{p} +  |v^\e_{s-} - v_{s-}|^p\big) \|\Delta_s \ti Z\|^2 \nonumber\\ 
&= \int_0^t \int_{\RR^r} \big(|u_{s-}^\e- u_{s-}|^{p} +  |v^\e_{s-} - v_{s-}|^p\big) \|z\|^2 \ti N'(dsdz) \label{stoch int 1}\\
&\qquad +\int_0^t \int_{\|z\|>1} \big(|u^\e_s- u_s|^{p} +  |v^\e_s - v_s|^p\big) \|z\|^2 \nu'(dz) ds.\label{stoch int 2}
\end{align}
The maximal inequality in \cite{SLB98} by Saint Loubert Bi\'e for the exponent $1$  yields
\begin{align*}
&\EE\Big[\sup_{t\in [0, T]} \int_0^t \int_{\RR^r} \big(|u_{s-}^\e- u_{s-}|^{p} +  |v^\e_{s-} - v_{s-}|^p\big) \|z\|^2 \ti N'(dsdz)\Big]\\
&\lqq C_{18} \int_0^T \int_{\RR^r} \EE\Big[|u^\e_s- u_s|^{p} +  |v^\e_s - v_s|^p\Big] \|z\|^2 \nu'(dz) ds\\
&\lqq C_{19} \int_{\RR^r} \|z\|^2 \nu'(dz) 
\Big(\int_0^T \EE\Big[\sup_{[0, s]} |u^\e- u|^{p}\Big]ds  +  \int_0^T \EE\Big[|v^\e_s - v_s|^p\Big] ds\Big)\\
&\lqq C_{20} 
\int_0^T \EE\Big[\sup_{t\in [0, s]} |u^\e- u|^{p}\Big]ds  +  C_{21} \e^p T^{p+1}.
\end{align*}
The term (\ref{stoch int 2}) is treated obviously such that 
\begin{align}
\EE[\sup_{[0, T]} |I_7|] \lqq \e^2 C_{22} \int_{\RR^r} \|z\|^2 \nu'(dz) 
\int_0^T \EE\Big[\sup_{[0, s]} |u^\e- u|^{p}\Big] ds  + C_{23} \e^p  T^{p+1}.\label{I7 fertig}
\end{align}
$\mathbf{I_8}$: For $I_8$ Lemma \ref{lem: invariance lemma}, statement 2), yields 
\begin{align*}
&\sum_{0< s\lqq t} |u_{s-}^\e- u_{s-}|^{p-2} \lgl u_{s-}^\e- u_{s-}, \Phi^{\e \ti \fK_H\Delta_s \ti Z}(v_{s-})- v_{s-}
 - \e \ti \fK_H(v_{s-})\Delta_s \ti Z\rgl \\
&\lqq \sum_{0< s\lqq t} |u_{s-}^\e- u_{s-}|^{p-1} |\Phi^{\e \ti \fK_H\Delta_s \ti Z}(v_{s-})- v_{s-}
 - \e \ti \fK_H(v_{s-})\Delta_s \ti Z| \\ 
&\lqq \e^2 C_{24}\sum_{0< s \lqq t} |u_{s-}^\e- u_{s-}|^{p-1} \|\Delta_s \ti Z\|^2,
\end{align*}
such that Saint Loubert Bi\'e's maximal inequality with exponent $1$ and 
elementary Young's estimate for parameters $\frac{p-1}{p}$ and $p$ yield 
\begin{align}
\EE[\sup_{[0, T]} |I_8|]
&\lqq  \e^2 C_{25} \int_{\RR^r} \|z\|^2 \nu'(dz) \int_{0}^T \EE\Big[\sup_{[0, s]} |u^\e- u|^{p-1}\Big] ds\nonumber\\
&\lqq  \e C_{26} \int_{0}^T \EE\Big[\sup_{[0, s]} |u^\e- u|^{p}\Big] ds + C_{26} \e^p T.\label{I8 fertig}
\end{align}

\paragraph{3. Estimate of $I_5$: }
\begin{align*}
\int_0^T |u_s^\e- u_s|^{p-2} \lgl u^\e_s- u_s, \e \fK_H(u_s, v_s)\rgl ds 
&\lqq C_{27} \int_0^T \e |u^\e_s- u_s|^{p-1} ds \\
&\lqq C_{27} \int_0^T \e |u^\e_s- u_s|^{p} ds + C_{27} \e^p T
\end{align*}
such that 
\begin{align}
\EE[\sup_{[0,T]} |I_5|] \lqq \e C_{27} \int_0^T \EE\Big[\sup_{[0,s]}|u^\e- u|^{p}\Big] ds+ C_{27} \e^p T.\label{I5 fertig}
\end{align}

\paragraph{4. Nonlinear comparison principle: } 
Taking  the supremum and the expectation of the left-hand side of equation (\ref{eq: pivot}) 
and combining the estimates of 
$\sum_{i=1}^8 \EE[\sup_{[0,T]} |I_i|]$ 
given by (\ref{I2 fertig}), (\ref{I2 comp}), (\ref{I6 fertig}), (\ref{I3 fertig}), (\ref{I7 fertig}), (\ref{I8 fertig}) 
and (\ref{I5 fertig}) we obtain a positive constant $C_{28}$
\begin{align*}
 \EE\left[\sup_{t\in [0, T]}|u_t^{\e} - u_t|^p\right] 
&\lqq C_{28} \bigg(\int_{0}^T \EE\left[\sup_{t\in [0, s]}|u^{\e} - u|^p\right] ds
+\bigg(\int_0^T \EE\left[\sup_{t\in [0, s]}|u^{\e} - u|^p\right]ds \bigg)^\frac{1}{2} 
+ \e^p T^p (1+T)\bigg).
\end{align*}
For the concave invertible function $G(x) = x +\sqrt{x}$, $\psi(t) := \EE\left[\sup_{t\in [0, T]}|u_t^{\e} - u_t|^p\right]$ 
and $p(t)= \e^p t^p(1+t)$ 
we have achieved the integral inequality 
\begin{align*}
\psi(t) \lqq C_{28} G\Big(\int_0^t \psi(s) ds\Big) + C_{28} p(t). 
\end{align*}
Hence the convex function $G^{-1}(x) = (x+\frac{1}{2}) - \sqrt{ x - \frac{1}{4}}$ yields for $C_{29} = \max\{2 C_{28}, 4 C_{28}^2\}$ 
\begin{align*}
G^{-1}(\psi(t)) 
&\lqq G^{-1}\bigg(C_{28} G\Big(\int_0^t \psi(s) ds\Big) + C_{28} p(t)\bigg)\\  
&\lqq \frac{1}{2} G^{-1}\bigg(2C_{28} G\Big(\int_0^t \psi(s) ds\Big)\bigg) + \frac{1}{2} G^{-1}\big(2C_{28} p(t)\big)\\
&\lqq C_{29} \int_0^t \psi(s) ds  + C_{29} G^{-1}\big(p(t)\big).
\end{align*}
Note that $G^{-1}(x) \gqq \kappa x$ for all $\kappa \in (0,1)$ and $x \gqq x_0$ for some $x_0>0$, which we calculate as follows. 
\begin{align*}
(x_0+\frac{1}{2}) - \sqrt{x_0+\frac{1}{4}} = \kappa x_0 
\qquad \vzv \qquad x_0  = \frac{\kappa}{(1-\kappa)^2}.
\end{align*}
The quadratic function $q x^2$ satisfying that $q x_0^2 = \kappa x_0$ 
has the prefactor 
\begin{align*}
q x_0^2 = \kappa x_0 \qquad \vzv \qquad q  = \frac{\kappa}{x_0} 
= (1-\kappa)^2,
\end{align*}
and we obtain that any $\kappa \in (0,1)$ and $x >0$ satisfy 
$G^{-1}(x) \gqq \min\{(1-\kappa)^2 x, \kappa x\}$. Hence 
\begin{align*}
\min\{(1-\kappa)^2 \psi^2(t), \kappa \psi(t)\} \lqq G^{-1}(\psi(t)) 
&\lqq C_{29} \int_0^t \psi(s) ds  + C_{29} G^{-1}\big(p(t)\big).
\end{align*}
This implies for $\psi(t) \gqq \kappa x_0 = \frac{\kappa^2}{(1-\kappa)^2}$ 
the integral inequality  
\begin{align*}
\psi(t) &\lqq \frac{C_{29}}{\kappa} \int_{t_0}^{t} \psi(s) ds  + \frac{C_{29}}{\kappa} G^{-1}\big(p(t)\big)
\end{align*}
for $t\gqq t_0$, where $t_0 = \inf\{t>0~|~\psi(t) = \kappa x_0\}$. Gronwall's inequality yields for $t\gqq t_0$ 
\begin{align*}
\psi(t) 
&\lqq \frac{C_{29}}{\kappa} G^{-1}\big(p(t)\big) \exp\Big((t-t_0) \frac{C_{29}}{\kappa}\Big). 
\end{align*}
Since the running supremum $t \mapsto \psi(t)$ is monotonically increasing we also obtain that for any $t\gqq 0$
\begin{align*}
\psi(t) \lqq \frac{C_{29}}{\kappa} G^{-1}\big(p(t) \vee p(t_0)\big) \exp\Big( \frac{C_{29}}{\kappa} t\Big).
\end{align*}
Taking into account that $G^{-1}(x) \lqq x^2$ for all $x\gqq 0$ we obtain for $\kappa = \frac{1}{2}$ and $C_{30} = 2 C_{29}$ 
\begin{align}
\EE\left[\sup_{[0, T]}|u^{\e} - u|^p\right] 
&\lqq C_{30} \e^{2p} (1+(T\vee t_0))^{2(p+1)} \exp\big( C_{30} T\big).\label{eq: Gu estimate}
\end{align}
\end{proof}

\paragraph{IV. End of the proof of Proposition \ref{lem: preliminary}} 
Eventually Minkowski's inequality, the Lipschitz estimate (\ref{eq: ungleichung 1}), the sum of 
the vertical (\ref{eq: vertical lem}) and the horizontal (\ref{eq: Gu estimate lem}) estimate 
yield the desired result for $\bar C_3 = \bar C_1 \vee \bar C_2$ and constants $k_1, k_2>0$ 
\begin{align*}
\EE\left[\sup_{t\in [0, T]}|h(X^\e_t(x_0) ) - h(X_t(x_0))|^p\right]
&\lqq \bar C_{3} \Big(\e^p (1+T^{2p+1}) +C_{30} \e^{2p} (1+(T\vee t_0))^{2(p+1)} \exp\big( C_{30} T\big)\Big)\\
&\lqq k_1 \e^{p} \exp\big(k_2 T\big).
\end{align*}
This finishes the proof. 
\end{proof}

\section{The averaging error}

For convenience we fix the following notation. 
Given $h: M\ra \RR^n$ a globally Lipschitz continuous function and 
$Q^h: V \ra \RR^n$ its average on the 
leaves defined as (\ref{def: average}). 
For $t\gqq 0$, $x_0\in M$ and $\e\in (0,1)$ denote the error term 
\[
\delta^h_{x_0}(\e, t) := \int_{0}^{t 
} 
h(X^\e_{\frac{s}{\e}}(x_0)) - Q^h(\pi(X^\e_{\frac{s}{\e}}(x_0))) ds.
\]

\begin{prp}\label{prp: componente vertical}
 Let the assumptions of Proposition \ref{lem: preliminary} be satisfied for a fixed $p\gqq 2$. 
Then for any $h: M\ra \RR^n$ globally Lipschitz continuous, $\la\in (0,1)$ and $x_0 \in M$ 
there are constants $b_1>0$ and $\e_0\in (0,1)$ such that for $\e\in (0, \e_0]$ and $T \in [0, 1]$ we have 
\begin{align*}
\left(\EE\left[\sup_{s \in [0, T]} |\delta_{x_0}^h(\e,s)|^p \right]\right)^\frac{1}{p} 
\lqq b_1 T \left[ \e^{\la}  +   \eta \left( c T |\ln \epsilon | \right) \right],
\end{align*}
where $c = c_{\la, p}$ is given by Corollary \ref{cor: preliminary} 
and $\eta$ is the ergodic rate of convergence 
given in equation (\ref{def: function eta}) of Hypothesis 3.
\end{prp}

\paragraph{Proof of Proposition (\ref{prp: componente vertical}): The common part : } 
Fix $x_0\in M$. For $\e\in (0,1)$ and $T> 0$ we define the partition 
\[
t_0 = 0 < t_1^\e < \dots < t_{N^\e}^\e \lqq \frac{T}{\e} 
\]
with the following step size 
\begin{equation*}
\Delta_\e := -c T \ln(\e),
\end{equation*}
where $c>0$ is given by Corollary \ref{cor: preliminary}. 
The grid points are defined as 
\[
t_n^\e := n \Delta_\e
\qquad \mbox{ for }\qquad 0 \lqq n \lqq N_\e, \quad \e \in (0,1)
\quad \mbox{ and the total number is }\qquad N_\e = \lfloor \frac{1}{c \e |\ln(\e)|} \rfloor +1. 
\]
We rewrite the first summand of $\delta^h$ by 
\begin{align*}
\int_{0}^{t 
} h(X^\e_{\frac{s}{\e}}(x_0)) ds &= \e 
\int_{0}^{\frac{t}{\e} 
} h(X^\e_{s}(x_0)) ds\\
& = \e \sum_{n=0}^{N_{\e}-1} \int_{t_n}^{t_{n+1}} h(X^\e_{s}(x_0)) ds + \e 
\int_{t_{N_\e}}^{\frac{t}{\e} 
}h(X^\e_{s}(x_0)) ds.
\end{align*}
We lighten notation and omit for convenience in the sequel the superscript $\e$ and $h$ 
as well as the initial value $x_0$ whenever possible.  
The triangle inequality yields 
\begin{equation}\label{eq: delta decomposition}
|\delta^h_{x_0}(\e, t)| \lqq |A_1(t, \e)| + |A_2(t, \e)| +|A_3(t, \e)|,
\end{equation}
where 
\begin{align*}
A_1(t, \e) &:= \e \sum_{n=0}^{N_\e-1} \int_{t_n}^{t_{n+1}} [h(X^\e_{s}(x_0)) -h(X_{s-t_n}(X^\e_{t_n}(x_0)))] ~ds, \\[2mm]
A_2(t, \e) &:= \e \sum_{n=0}^{N_\e-1} \int_{t_n}^{t_{n+1}} 
[h(X_{s-t_n}(X^\e_{t_n}(x_0))) - Q(\pi(X^\e_{t_n}(x_0)))] 
~ds, \\[2mm]
A_3(t, \e) &:=\sum_{n=0}^{N_\e-1}\e \Delta_\e Q(\pi(X^\e_{t_n}(x_0))) - \int_{0}^{t_{N_\e}}  Q(\pi(X^\e_{\frac{s}{\e}}(x_0))) ~ds.
\end{align*}
The following three lemmas estimate the preceding terms. This being done the proof of Proposition~\ref{prp: componente vertical} 
is finished. 

\begin{lem}\label{Lemma1}
For any $\la \in (0,1)$ and $x_0\in M$ there are $b_2>0$ and $\e_0\in (0,1)$ such that for any $\e\in (0, \e_0]$ and $T\in [0,1]$ 
\begin{align*}
\left(\EE\left[\sup_{s\in [0, T]} |A_1(s, \e)|^p \right]\right)^{\frac{1}{p}} \lqq b_2 T \e^{\la}.
\end{align*}
\end{lem}
\begin{proof} 
Fix $\la \in (0,1)$. We apply the Markov property of $X_t$ and H\"older's inequality in the time variable 
and the fact that by definition $N_\e \Delta_\e \lqq T$ 
\begin{align*}
\EE\left[\sup_{s\in [0, T]} |A_1(s, \e)|^p\right]^\frac{1}{p} 
&=\e \sum_{n=0}^{N_\e-1} \EE\Big[\EE\Big[
\big|\int_{t_n}^{t_{n+1}} [h(X^\e_{s}(x_0)) -h(X_{s-t_n}(X^\e_{t_n}(x_0)))] ~ds\big|^p~|~\fF_{t_n}\Big]\Big]^\frac{1}{p}\\
&= \e \sum_{n=0}^{N_\e-1} \EE\Big[\EE\Big[
\big|\int_{0}^{t_{1}} [h(X^\e_{s-t_n}(y)) - h(X_{s-t_n}(y))] ~ds\big|^p
~|~y = X^\e_{t_n}(x_0)\Big]\Big]^\frac{1}{p}\\
&\lqq \e N_\e \Delta_\e 
\EE\Big[\EE\Big[\sup_{s\in [0, t_{1}]} \big|h(X^\e_{s}(X^\e_{t_n}(x_0))) -h(X_{s}(X^\e_{t_n}(x_0)))\big|^p ~|~y = X^\e_{t_n}(x_0)\Big] \Big]^\frac{1}{p}\\
&\lqq T \max_{n\in \{1,\dots, N_\e\}} 
\EE\Big[\EE\Big[\sup_{s\in [0, t_{1}]} \big|h(X^\e_{s}(X^\e_{t_n}(x_0))) -h(X_{s}(X^\e_{t_n}(x_0)))\big|^p ~|~y = X^\e_{t_n}(x_0)\Big]\Big]^\frac{1}{p}\\
&\lqq T \max_{n\in \{1,\dots, N_\e\}} \EE[k_3(\diam(L_{X^\e_{t_n}(x_0)}))] \e^\la.
\end{align*}
The last estimate is an application of  Corollary \ref{cor: preliminary}, 
where we have found that 
the value $k_3 = k_3(d_{v_0})$, $d_{v_0} = \diam(L_{v_0})$ from (\ref{eq: cor preliminary}) 
is a polynomial in $d_{v_0}$ of order $p$ with positive coefficients. 

\noindent First we estimate the term involving $k_3$. 
Since by Hypothesis 1 the mapping $v_0 \mapsto d_{v_0}$ 
is Lipschitz continuous there is a Lipschitz constant $\ell'$ on $U$, 
which yields the constants $\al_1$ and $\al_2$ such that almost surely 
\begin{align}
k_3(d_{X^\e_{t_n}}) 
&\lqq k_3\big(d_{v_0} + \ell' \dist(L_{x_0}, L_{X^\e_{t_n}(x_0)})\big) \nonumber\\
&\lqq \al_1 (k_3(d_{v_0}) + 1) + \al_2 \dist(L_{x_0}, L_{X^\e_{t_n}(x_0)})^p.\label{eq: first step}
\end{align}
The next step consists in the estimate of the term $\EE[\dist(L_{x_0}, L_{X^\e_{t_n}(x_0)})^p]$. 
Using $X_t(x_0) \in L_{x_0}$ for all $t\gqq 0$ 
Lemma \ref{lem: vertical} ensures the existence of the positive constant 
$\bar C_1(d_{v_0})$ which is an affine function in $d_{v_0}$ such that 
for $k = C_0 \bar C_1(d_{v_0})$ 
\begin{equation}\label{eq: anfangswert}
\EE[\sup_{t\in [0, t_1]} \dist(L_{x_0}, L_{X^\e_{t}(x_0)})^p]^\frac{1}{p} \lqq 
C_0 \EE\Big[\sup_{s\in [0, t_1]} |v_s^\e(x_0) - v_s(x_0)|^p\Big]^\frac{1}{p}\lqq k \e^\la,
\end{equation}
where $C_0$ is the Lipschitz constant of the local coordinates $\phi$ with Lipschitz constant $\ell''$. 
We obtain 
\begin{align*}
\EE[\sup_{t\in [0, t_2]} \dist(L_{v_0}, L_{X^\e_{t}(x_0)})^p]^\frac{1}{p}
&\lqq \EE[\sup_{t\in [0, t_1]} \dist(L_{v_0}, L_{X^\e_{t}(x_0)})^p]^\frac{1}{p} 
+\EE[\EE[\sup_{t\in [t_1, t_2]} \dist(L_{v_0}, L_{X^\e_{t}(x_0)})^p~|~\fF_{t_1}]]^\frac{1}{p}.
\end{align*}
The first term on the right-hand side obeys (\ref{eq: anfangswert}), the second one 
can be calculated recursively 
\begin{align*}
\EE[\EE[\sup_{t\in [t_1, t_2]} \dist(L_{v_0}, L_{X^\e_{t}(x_0)})^p~|~\fF_{t_1}]]^\frac{1}{p}
&\lqq \EE[\EE[\sup_{t\in [0, t_1]} \dist(L_{x_0}, L_{X^\e_{t}(y)})^p~|~y = X^\e_{t_1}(x_0) ]]^\frac{1}{p}\\
&\lqq C_0 \EE[\bar C_1(d_{X^\e_{t_1}(x_0)})]\e^\la \\
&\lqq (k+ \ell'' k \e^\la) \e^\la.
\end{align*}
Hence 
\begin{align*}
\EE[\sup_{t\in [0, t_2]} \dist(L_{v_0}, L_{X^\e_{t}(x_0)})^p]^\frac{1}{p}
&\lqq \frac{k}{\ell''} (2 \ell'' \e^\la + (\ell''\e^{\la})^2). 
\end{align*}
We argue by induction  
\begin{align*}
&\EE[\sup_{t\in [0, t_{n+1}]} \dist(L_{v_0}, L_{X^\e_{t}(x_0)})^p]^\frac{1}{p} \\
&\lqq \EE[\sup_{t\in [0, t_{n}]} \dist(L_{v_0}, L_{X^\e_{t}(x_0)})^p]^\frac{1}{p} + 
\EE[\EE[\sup_{t\in [0, t_{1}]} \dist(L_{v_0}, L_{X^\e_{t}(y)})^p~|~y = X^\e_{t_n}(x_0)]]^\frac{1}{p} \\ 
&\lqq \frac{k}{\ell''} \bigg(\Big(\sum_{i=0}^n \binom{n}{i} (\ell'' \e^\la)^i -1\Big)  + 
C_0 \EE[\bar C_1(d_{X^\e_{t_n}(x_0)})] \e^\la.
\end{align*}
We continue with the second term 
\begin{align*}
C_0 \EE[\bar C_1(d_{X^\e_{t_n}(x_0)})] \e^\la
&\lqq 
C_0 \bar C_1(d_{v_0}) \e^\la + \ell'' \EE[\sup_{t\in [0, t_{n}]} \dist(L_{v_0}, L_{X^\e_{t}(x_0)})] \e^\la 
\\
&\lqq 
k \e^\la + \ell'' \EE[\sup_{t\in [0, t_{n}]} \dist(L_{v_0}, L_{X^\e_{t}(x_0)})^p]^\frac{1}{p} \e^\la 
\\
&\lqq 
\frac{k}{\ell''} \ell'' + \ell''\e^\la \frac{k}{\ell''}\Big(\sum_{i=0}^n \binom{n}{i} (\ell'' \e^\la)^i -1\Big)\bigg) \\
\end{align*}
eventually leading to  
\begin{align*}
\EE[\sup_{t\in [0, t_{N_\e+1}]} \dist(L_{v_0}, L_{X^\e_{t}(x_0)})^p]^\frac{1}{p} 
&\lqq \frac{k}{\ell''} \bigg(\Big(\sum_{i=0}^{N_\e} \binom{N_\e}{i} (\ell'' \e^\la)^i -1\Big)  + 
\frac{k}{\ell''}\Big(\sum_{i=1}^{N_\e+1} \binom{N_\e}{i+1} (\ell'' \e^\la)^i -1\Big)\bigg)\\
&= \frac{k}{\ell''} \bigg(\sum_{i=0}^{N_\e+1} \binom{N_\e+1}{i}(\ell'' \e^\la)^i -1\bigg)\\
&= \frac{k}{\ell''}\big((1+\ell'' \e^\la)^{N_\e}-1\big) 
\lqq k N_\e  \e^\la
\end{align*}
for $\e \in (0, \e_0]$ sufficiently small. 
Combining the preceding result with (\ref{eq: first step}) yields for $\e \in (0, \e_0]$ 
\begin{align*}
\EE\left[\sup_{s\in [0, T]} |A_1(s, \e)|^p\right]^\frac{1}{p} 
&\lqq T \max_{n\in \{1,\dots, N_\e\}} \EE[k_3(\diam(L_{X^\e_{t_n}(x_0)}))] \e^\la \\
&\lqq T \e^\la \Big(\al_1 (k_3(d_{v_0}) + 1) + \al_2 \EE[\sup_{t\in [0, t_n]} \dist(L_{x_0}, L_{X^\e_{t}(x_0)})^p]\Big) \\
&\lqq T \e^\la \Big(\al_1 (k_3(d_{v_0}) + 1) + \al_2 (k N_\e  \e^\la)^p\Big),
\end{align*}
and a positive constant $b_2$ which yields the desired result for any $\e \in (0, \e_0]$ and $\ti \la \in (0,\la - p(1-\la))$ 
\begin{align*}
\EE\left[\sup_{s\in [0, T]} |A_1(s, \e)|^p\right]^\frac{1}{p} 
\lqq b_2 T \e^{\ti \la}.
\end{align*}
\end{proof}

\begin{lem}\label{Lemma2}Let $\eta$ be the rate of convergence defined in (\ref{def: function eta}). 
For the process $A_2$ in inequality (\ref{eq: delta decomposition}), $T>0$ fixed and 
$\la \in (0,1)$ here is $c = c(\la) \in (0,1)$ 
for which there are $\e_0\in (0,1)$ and $b_3>0$ such that for any $T\in [0,1]$ and $\e \in (0, \e_0]$   we have 
\begin{align*}
\left(\EE\left[\sup_{s\in [0, T]} |A_2(s, \e)|^p \right]\right)^{\frac{1}{p}} \lqq 
 b_3 T \ \eta \left( c T |\ln \epsilon |\right). 
\end{align*}
\end{lem}

\noindent The proof is virtually identical to the proof of Lemma 3.3 in \cite{GR13} in the purely Brownian case. 

\begin{lem}\label{Lemma3}
Due to Hypothesis 2 $Q^h$ is globally Lipschitz continuous. 
Then the process $A_3$ in inequality (\ref{eq: delta decomposition}) satisfies the following. 
For any $\la\in (0,1)$ there are $\e_0\in (0,1)$ and $b_4>0$ 
such that for any $\e \in (0, \e_0]$ and $T\in [0,1]$ 
\begin{align*} 
\left(\EE\left[\sup_{s\in [0, T]} |A_3(s,\e)|^p \right]\right)^{\frac{1}{p}} \lqq 
b_4 T \e^\la.
\end{align*}
\end{lem}

\begin{proof} 
We calculate 
\begin{align}
|A_3(T,\e)| &=  \Big| \sum_{n=0}^{N_\e-1} \e \Delta_\e Q(\pi(X^\e_{t_n}))
- \int_{0}^{N_\e \Delta_\e}  Q(\pi(X^\e_{\frac{s}{\e}})) ~ds \Big|\nonumber\\
& \lqq \e \sum_{n=0}^{N_\e-1} \Delta_\e \sup_{\e t_n \lqq s < \e t_{n+1}}  
|Q(\pi(X^\e_s)) - Q(\pi(X^\e_{t_n}))|\nonumber\\
& \lqq \e \Delta_\e C \sum_{n=0}^{N_\e-1} \sup_{t_n \lqq s < t_{n+1}} |v^\e_s - v^\e_{t_n}|. \label{eq: A3 estimate}
\end{align}
Minkowski's inequality, the Markov property, 
Lemma \ref{lem: vertical} with $\bar C_1$ in $d_{v_0}$ 
and estimate (\ref{eq: anfangswert}) lead to 
\begin{align*}
\EE[\sup_{s\in [0, T]} |A_3(s,\e)|^p]^\frac{1}{p}
& \lqq T C \max_{n\in \{1, \dots, N_\e\}} 
\EE[\EE[\sup_{t_n \lqq s < t_{n+1}} |v^\e_{s-t_{n}}(y) - v^\e_{0}(y)|^p~|~y = X_{t_n}(x_0)]]^\frac{1}{p}\\
& \lqq T C \max_{n\in \{1, \dots, N_\e\}} 
\EE[\EE[\sup_{t_0 \lqq s < t_{1}} |v^\e_{s}(y) - v^\e_{0}(y)|^p~|~y = X_{t_n}(x_0)]]^\frac{1}{p}\\
& \lqq T C \EE\Big[\bar C_1(\diam(L_{X_{t_{N_\e}}(x_0)}))\Big] \e^\la \\
& \lqq T C\bar C_1(d_{v_0}) (\e^\la + \ell N_\e \e^{2\la})\\
&\lqq T b_4 \e^{\ti \la}
\end{align*}
for $\ti \la \in (0, 2\la-1)$. 
\end{proof}

\noindent \textbf{(Proof of Proposition \ref{prp: componente vertical}:} 
Combining Minkowski's inequality with Lemma \ref{Lemma1}-\ref{Lemma3} yields 
for any $\la\in (0,1)$ and $p\gqq 2$ 
constants $c_{\la, p}$, $\e_0, k_0\in (0,1)$ and $b_1>0$ 
such that for any $T\in [0,1]$ satisfying $\e_0 T < k_0$ 
$\e\in (0, \e_0]$ implies 
\begin{align*}
\left(\EE\left[\sup_{t\in [0, T]} |\delta_{x_0}^h(\e, t)|^p\right]\right)^{\frac{1}{p}} 
&\lqq  b_1\, T \Big( \e^{\la}  +  \eta(c T |\ln(\epsilon)| \Big). 
\end{align*}

\section{Proof of the main result} \label{sec: proof of the main result}

We keep the notation of the Proof of Proposition \ref{prp: componente vertical}. 
By the change of variable formula for canonical Marcus integrals, 
\cite{KPP95} Proposition 4.2,  we may rewrite (\ref{eq: transversale SDE in koordinaten})
\begin{align*}
v^\e_{\frac{t}{\e}
} 
&= \int_0^{t 
} \fK_V(u^\e_{\frac{s}{\e}-}, u^\e_{\frac{s}{\e}-}) ds + \int_0^{t 
} (\ti \fK_V)(v^\e_{\frac{s}{\e}-}) \diamond d \ti Z_s.
\end{align*}
Since equation (\ref{def: w}) tells us that $w \big(t\big) = \int_0^t Q^{\fK_V}(w(s)) ds + \int_0^t \ti \fK_V(w(s-)) \diamond d\ti Z_s$ 
we obtain 
\begin{align*}
&v^\e_{\frac{t}{\e}
} - w \big(t\big) \\
& = \int_0^{t
} \fK_V(u^\e_\frac{s}{\e},v^\e_\frac{s}{\e}) - Q^{\fK_V}(w(s)) ds + 
\int_{0}^{t
} \big(\ti \fK_V(v^\e_{\frac{s}{\e}-}) - \ti \fK_V(w(s-))\big) \diamond d\ti Z_s\\
& = \delta_{x_0}^{\fK_V}(\e, t) + \int_0^{t} Q^{\fK_V}(u^\e_\frac{s}{\e}, v^\e_\frac{s}{\e}) - Q^{\fK_V}(w(s)) ds 
+ \int_{0}^{t} 
\big(\ti \fK_V(v^\e_{\frac{s}{\e}-}) - \ti \fK_V( w(s-))\big) d\ti Z_s\\
&\quad +\sum_{0< s \lqq t} 
\Big(\Phi^{\ti \fK_V\Delta_{s-} \ti Z}(v^\e_{\frac{s-}{\e}}) - \Phi^{\ti \fK_V \Delta_s \ti Z}(w(s-)) - (v^\e_{\frac{s-}{\e}}-w(s-)) - 
\big(\ti \fK_V(v^\e_{\frac{s-}{\e}})- \ti K_V(w(s-)) \big)\Delta_s \ti Z\Big)\\
&= \delta_{x_0}^{\ti \fK_V}(\e, t) + O_1 + O_2 + O_3.
\end{align*}
Now, for $T\in [0, 1]$ we obtain with the help of Jensen's inequality 
\begin{align*}
\EE\Big[\sup_{[0, T]} |O_1|^p\Big] 
&\lqq a_1 T^{p-1}\int_0^T \EE\Big[\sup_{t\in [0, s]} |v^\e_\frac{t}{\e}-w(t)|^p\Big]ds.
\end{align*}
Kunita's maximal inequality \cite{Ku04} yields a positive constant $a_2 = a_2(p)$ such that 
for the constant $a_3 = a_2 \ti \ell (\int_{\RR^r} \|z\|^2 \nu'(dz)^{p/2} + 2\int_{\RR^r} \|z\|^p \nu'(dz))$ 
we have 
\begin{align*}
\EE\Big[\sup_{[0, T
]} |O_2|^p\Big] 
&\lqq \EE\Big[\sup_{t\in [0, T]} |\int_{0}^{t
} \int_{\RR^r}\big(\ti \fK_V(v^\e_{\frac{s}{\e}-}) - \ti \fK_V( w(s-))\big)z \ti N'(dsdz)|^p\Big]\\
&\qquad +\EE\Big[\sup_{t\in [0, T]} |\int_{0}^{t
} \int_{\|z\|>1}\big(\ti \fK_V(v^\e_{\frac{s}{\e}}) - \ti \fK_V( w(s))\big)z \nu'(dz) ds|^p\Big]\\ 
&\lqq a_2 \EE\Big[\int_0^{T
}\int_{\RR^r} | \big(\ti \fK_V(v^\e_{\frac{s}{\e}}) - \ti \fK_V( w(s))\big)z|^p \nu'(dz) ds\Big] \\
&\qquad +a_2 \EE\Big[\Big(\int_0^{T
}\int_{\RR^r} | \big(\ti \fK_V(v^\e_{\frac{s}{\e}}) - 
\ti \fK_V( w(s))\big)z|^2 \nu'(dz) ds\Big)^\frac{p}{2}\Big]\\
&\qquad +\int_{\|z\|>1} \|z\|^p \nu'(dz) 
\EE\Big[\int_{0}^{T
} \big| \ti \fK_V(v^\e_{\frac{s}{\e}}) - \ti \fK_V( w(s))\big|^p ds\Big]\\ 
&\lqq a_3 \int_0^{T} \EE\Big[\sup_{s\in [0, t]} |v^\e_{\frac{s}{\e}
}-w(s)|^p\Big] dt.
\end{align*}
Finally the Lipschitz continuity of the vector fields $\ti \fK_V$ and $(D \ti \fK_V) \ti \fK_V$ 
and Lemma~\ref{lem: invariance lemma} provide a constant $a_4>0$ 
\begin{align}
| O_3|^p &\lqq  \Big|\sum_{0< s\lqq T
} \Phi^{\ti\fK_V \Delta_s \ti Z}(v^\e_{\frac{s-}{\e}}) 
-\Phi^{\ti \fK_V \Delta_s \ti Z}(w_{s-}) \nonumber \\
&\qquad\qquad  -(v^{\e}_{\frac{s-}{\e}} -w(s-)) 
- (\e \ti \fK_V(v^\e_{\frac{s-}{\e}})-\e \ti \fK_V(w({s-}))) \Delta_s \ti Z\Big|^p\nonumber\\[3mm]
& \qquad \lqq  \Big(\frac{1}{2} \sum_{0< s\lqq T} |(D \ti \fK_V(v^{\e}_{\frac{s-}{\e}})\Delta_s \ti Z) 
\ti \fK_V(v^{\e}_{\frac{s-}{\e}})\Delta_s \ti Z 
-(D \ti \fK_V(w(s-))\Delta_s \ti Z) \ti \fK_V(w(s-))\Delta_s \ti Z| \Big)^{p}\nonumber\\[2mm]
& \qquad \lqq  \Big(a_4 \sum_{0< s\lqq T} |v^{\e}_{\frac{s-}{\e}} -w(s-)| \|\Delta_s \ti Z\|^2 \Big)^{p}\nonumber\\[2mm]
& \qquad \lqq (a_4)^p \Big(\sum_{0< s\lqq T} |v^{\e}_{\frac{s-}{\e}} -w(s-)| \|\Delta_s \ti Z\|^2 \Big)^{p}.
\label{eq: Marcus estimate J3}
\end{align}
Switching to the representation in terms of the Poisson random measure  
\begin{align}
&\sum_{0< s\lqq T} |v^{\e}_{\frac{s-}{\e}} -w(s-) \|\Delta_s \ti Z\|^2 \nonumber\\
&= \int_0^T \int_{\RR^r} |v^{\e}_{\frac{s-}{\e}} -w(s-)| \|z\|^2 \ti N'(dsdz) 
+\int_0^T \int_{\|z\|>1} |v^{\e}_{\frac{s}{\e}} -w(s)|  \|z\|^2 ~\nu'(dz)~ds, \label{eq: Marcus estimate 1st term}
\end{align}
we obtain
\begin{align}
\EE[\sup_{[0, T]} |O_3|^p] 
&\lqq 2^{p-1}(a_4)^p \bigg(\EE\Big[\sup_{t\in [0,T]} \big|\int_0^t \int_{\RR^r} |v^{\e}_{\frac{s-}{\e}} -w(s-)| \|z\|^2 \ti N'(dsdz)\big|^p\Big] &\nonumber\\
& +\EE\Big[\big|\int_0^T \int_{\|z\|>1} |v^{\e}_{\frac{s}{\e}} -w(s)| \|z\|^2 ~\nu'(dz)~ds\big|^p\Big]\bigg)\nonumber\\
& =: O_4 + O_5. \label{eq: major estimate}
\end{align}
We apply Kunita's maximal inequality \cite{Ap09} which yields a constant $a_5 = a_5(p)>0$ and Jensen's inequality. 
For $a_6 = 2^{p-1}(a_4 a_5)^p$ and $a_7 = a_6 (\int_{\RR^r}  \|z\|^{2p} \nu'(dz) + (\int_{\RR^r} \|z\|^{4} \nu'(dz))^\frac{p}{2}$ 
we hence obtain 
\begin{align}
O_4 &\lqq a_6  \bigg( \EE\Big[\int_0^t \int_{\RR^r} |v^{\e}_{\frac{s}{\e}} -w(s)|^p  \|z\|^{2p} \nu'(dz)ds\Big] \nonumber\\
&\qquad + \EE\Big[\Big(\int_0^t \int_{\RR^r} |v^{\e}_{\frac{s}{\e}} -w(s)|^2 \|z\|^{4} \nu'(dz)ds\Big)^\frac{p}{2}\Big]\bigg) \nonumber\\
&= a_6 \bigg\{ \int_{\RR^r}  \|z\|^{2p} \nu'(dz)\;\EE\Big[ \int_0^t |v^{\e}_{\frac{s}{\e}} -w(s)|^p ds\Big] \nonumber\\
&\qquad + \Big(\int_{\RR^r} \|z\|^{4} \nu'(dz)\Big)^\frac{p}{2}\; 
\EE\Big[\Big(\int_0^t|v^{\e}_{\frac{s}{\e}} -w(s)|^2 ds\Big)^\frac{p}{2} \Big]\bigg\} \nonumber\\
&\lqq a_7(T^{\frac{p}{2}-1}+1) \int_0^T \EE\left[ \sup_{s\in [0, t]} |v^{\e}_{\frac{s}{\e}} -w(s)|^p  \right] dt. \label{eq: J1}
\end{align}
The term $O_5$ follows straight forward. 
Summing up $\sum_{i=1}^5 \EE[ \sup_{t\in [0, T]} |O_i|^p ]$ 
we obtain a constant $a_8$ such that for $T\in [0, 1]$ 
\begin{align*}
\EE\Big[\sup_{t\in [0, T]} |v^\e_{\frac{t}{\e}
} - w \big(t\big)|^p\Big] 
&\lqq a_8 
\int_0^T \EE\left[ \sup_{s\in [0, t]} |v^{\e}_{\frac{s}{\e}} -w(s)|^p  \right] dt 
+ \EE\Big[\delta_{x_0}^{\ti \fK_V}(T, \e)^p\Big].
\end{align*}
The standard nonautonomous version of Gronwall's lemma
implies a constant $a_9>0$ 
such that for $T\in [0, 1]$ 
\begin{align*}
\EE\left[ \sup_{s\in [0, T]} |v^{\e}_{\frac{s-}{\e}} -w(s-)|^p  \right]
&\lqq \EE\Big[\delta_{x_0}^{\fK_V}(T, \e)^p\Big] \exp(a_9 T) 
\lqq a_{10} \EE\Big[\delta_{x_0}^{\fK_V}(T, \e)^p\Big].
\end{align*}
Finally an application of Proposition \ref{prp: componente vertical} for $h = \fK_V$ finishes the proof of 
Theorem \ref{thms: main result 1}.

\section{Appendix }
\subsection{Detailed calculations of the example }
\label{subsec: appendix1}

In the sequel we verify (\ref{eq: convergencia Student L2}). Keeping in mind 
$Q^{\pi_r K}(x_0) = \frac{a+d}{2}r$ 
we obtain 
\begin{align}
&\EE\Big[\frac{1}{t}\int_0^t \pi_r K(X_s) ds\Big]\nonumber\\
&= \frac{r}{t}\int_0^t 
\EE\Big[\big(a \sin^2(Z_s) + d \cos^2(Z_s) + (b+c) \sin(Z_s)\cos(Z_s)\big)
\Big]ds\label{eq: exL1}\\
&= \frac{r}{t} a t + \frac{r}{t}\int_0^t 
\EE\Big[\big((d-a) \cos^2(Z_s) + (b+c) \sin(Z_s)\cos(Z_s)\big)
\Big]ds\nonumber\\
&= ra + \frac{r}{t} 
(d-a) \int_0^t \EE\Big[\cos^2(Z_s)\Big]ds  + (b+c) \int_0^t \EE[\sin(Z_s)\cos(Z_s)]ds\nonumber\\
&= ra + \frac{r}{t} 
(d-a) \int_0^t \frac{1}{2} \Re \EE\Big[\exp(i 2 Z_s)\Big]ds + \frac{d-a}{2} r \nonumber\\
&= \frac{a+d}{2} r +  
\frac{d-a}{2}\frac{r}{t} \int_0^t \Re \exp(-s \Psi(2))\;ds\nonumber\\
&= \frac{a+d}{2} r +  
\frac{d-a}{2}\frac{r}{t} \int_0^t \exp(-s \Re \Psi(2))\cos(s \Re \Psi(2))\;ds.\nonumber
\end{align}
Since $\Re\Psi(2)> 0$ we have 
\begin{align*}
\EE\Big[\Big|\frac{1}{t}\int_0^t \pi_r K(Z_s) ds-\frac{a+d}{2} r\Big|\Big]
\lqq \frac{|d-a|}{2}\frac{r}{t} \int_0^t \exp(-s  \Psi(2))\;ds
= \frac{|d-a|}{2 \Psi(2)} \frac{r}{t} \ra 0, \quad t\ra \infty. 
\end{align*}
This shows the result for $p=1$. For $p=2$ we continue 
\begin{align*}
&\EE\Big[\Big|\frac{1}{t}\int_0^t \pi_r K(Z_s) ds- \frac{a+d}{2}r\Big|^2\Big]\nonumber\\
&= \EE\Big[\Big|\frac{1}{t}\int_0^t 
 r \big(a \sin^2(Z_s) + d \cos^2(Z_s) + (b+c) \sin(Z_s)\cos(Z_s)\big)
ds- \frac{a+d}{2}r\Big|^2\Big]\nonumber\\
&= r^2 \EE\Big[\Big(\frac{1}{t}\int_0^t 
 \big(a \sin^2(Z_s) + d \cos^2(Z_s) + (b+c) \sin(Z_s)\cos(Z_s)\big)
ds\Big)^2\\
&\qquad - (a+d) \frac{1}{t}\int_0^t 
 \big(a \sin^2(Z_s) + d \cos^2(Z_s) + (b+c) \sin(Z_s)\cos(Z_s)\big)
ds+ \Big(\frac{a+d}{2}\Big)^2\Big]\nonumber\\
&= r^2 \bigg(\EE\Big[\Big(\frac{1}{t}\int_0^t 
 \big(a \sin^2(Z_s) + d \cos^2(Z_s) + (b+c) \sin(Z_s)\cos(Z_s)\big)
ds\Big)^2\Big]\\
&\qquad - (a+d) \underbrace{\EE\Big[\frac{1}{t}\int_0^t 
 \big(a \sin^2(Z_s) + d \cos^2(Z_s) + (b+c) \sin(Z_s)\cos(Z_s)\big)
ds\Big]}_{\ra \frac{a+d}{2}, \quad t\ra \infty \mbox{\footnotesize{ by (\ref{eq: exL1})}}} + \Big(\frac{a+d}{2}\Big)^2\bigg).\nonumber
\end{align*}
We calculate directly
\begin{align*}
E_t = &\EE\Big[\Big(\frac{1}{t}\int_0^t 
 \big(a \sin^2(Z_s) + d \cos^2(Z_s) + (b+c) \sin(Z_s)\cos(Z_s)\big)
ds\Big)^2\Big]\\
&= \frac{1}{t^2} 
\int_0^t \int_0^t 
\Big( a^2 \EE\Big[ a^2 \sin^2(Z_s) \sin^2(Z_\si)\Big] + d^2 \EE\Big[\cos^2(Z_s) \cos^2(Z_\si)\Big] \\
&\qquad \qquad + (b+c)^2\EE\Big[\sin(Z_s)\cos(Z_s)\sin(Z_\si)\cos(Z_\si)\Big] \\
&\qquad \qquad + a d \EE\Big[\sin^2(Z_s)\cos^2(Z_\si)] +a d \EE\Big[\cos^2(Z_s)\sin^2(Z_\si)\Big] \\
&\qquad \qquad + a (b+c) \EE\Big[\sin^2(Z_s) \sin(Z_\si)\cos(Z_\si)\Big]+a (b+c) \EE\Big[\sin(Z_s)\cos(Z_s)\sin^2(Z_\si)\Big] \\
&\qquad \qquad + d(b+c) \EE\Big[\cos^2(Z_s) \sin(Z_\si)\cos(Z_\si)\Big] + d(b+c) \EE\Big[\sin(Z_s)\cos(Z_s) \cos^2(Z_\si)\Big]\Big) \;d\si ds.
\end{align*}
We apply the elementary identities 
\begin{align*}
\cos(x)^2 \cos^2(y) &= \frac{1}{8}\Big(\cos(2(x-y))+\cos(2(x+y)) +2 \cos(2x) + 2\cos(2y) +2\Big)\\
\sin(x)^2 \sin^2(y) &= \frac{1}{8}\Big(\cos(2(x-y))+\cos(2(x+y)) -2 \cos(2x) - 2\cos(2y) +2\Big)\\
\sin(x)^2 \cos^2(y) &= -\frac{1}{8}\Big(\cos(2(x-y))+\cos(2(x+y)) +2 \cos(2x) - 2\cos(2y) -2\Big)\\
\sin(x)^2 \sin(y) \cos(y) &= \frac{1}{8}\Big(\sin(2(x-y))+\sin(2(x+y)) + 2\sin(y) +2\Big)\\
\cos(x)^2 \sin(y) \cos(y) &= -\frac{1}{8}\Big(\sin(2(x-y))+\sin(2(x+y)) - 2\sin(y) +2\Big)
\end{align*}
and obtain 
\begin{align*}
E_t 
&= \frac{1}{t^2} 
\int_0^t \int_0^t 
\frac{a^2}{8} \Big\{ \Re\EE\Big[e^{i 2(Z_s-Z_\si)}\Big] + \Re\EE\Big[e^{i 2(Z_s+Z_\si)}\Big] 
-2  \Re\EE\Big[e^{i 2Z_s}\Big] -2 \Re \EE[e^{i 2Z_\si}]+2\Big\}\\ 
&\qquad \qquad+ \frac{d^2}{8} \Big\{\Re \EE\Big[e^{i2(Z_s-Z_\si)}\Big] + \Re\EE\Big[e^{i 2(Z_s+Z_\si)}\Big]
+ 2 \Re \EE\Big[e^{i 2Z_s}\Big] +2 \Re \EE\Big[e^{ 2Z_\si}\Big]+2\Big\} \\
&\qquad \qquad + \frac{(b+c)^2}{8}\Big(\Re\EE\Big[e^{i 2(Z_s-Z_\si)}\Big]-\Re\EE\Big[e^{i 2(Z_s+Z_\si)}\Big] \\
&\qquad \qquad - \frac{a d}{4} \Big(\Re \EE\Big[e^{i2(Z_s-Z_\si)}\Big] + \Re\EE\Big[e^{i 2(Z_s+Z_\si)}\Big] 
+2\Re \EE\Big[e^{2Z_s}\Big] -2 \Re \EE\Big[e^{i2Z_\si}\Big]-2\Big) \\
&\qquad \qquad + \frac{a (b+c)}{4} \Big\{\Im \EE\Big[e^{i2(Z_s-Z_\si)}\Big]-\Im \EE\Big[e^{i2(Z_s+Z_\si)}\Big] + 2\Im \EE\Big[e^{2Z_\si}\Big]\Big]\Big\}\\
&\qquad \qquad - \frac{d(b+c)}{4} \Big\{\Im\EE\Big[e^{i 2(Z_s-Z_\si)}\Big]-\Im \EE\Big[e^{i 2(Z_s+Z_\si)}\Big] - 2\Im \EE\Big[e^{i2Z_\si}\Big] \Big\} \;d\si ds.
\end{align*}
Since
\begin{align*}
\EE\Big[e^{i2(Z_s\pm Z_\si)}\Big] = \EE\Big[e^{i2 Z_{s\pm \si}}\Big] = e^{-(s\pm \si) \Psi(2)}\qquad \mbox{ and } \qquad \EE\Big[e^{i4 Z_s}\Big] = e^{-s\Psi(4)}
\end{align*}
for $s\gqq \si$ we get 
\begin{align*}
E_t &= \frac{1}{t^2} 
\int_0^t \int_0^t 
\frac{a^2}{8} \Big\{ \exp(-|s-\si| \Re \Psi(2))\cos(|s-\si| \Im \Psi(2)) +  
\exp(-(s+\si) \Re \Psi(2)) \cos((s+\si) \Re \Psi(2)) \\
&\qquad \qquad \qquad -2  \exp(-s \Re \Psi(2)) \cos(s \Im \Psi(2)) -2 \exp(-\si \Re \Psi(2)) \cos(\si \Im \Psi(2))
+ 2\Big\}\\ 
&\qquad \qquad+ \frac{d^2}{8} \Big\{ \exp(-|s-\si| \Re \Psi(2)) \cos(|s-\si| \Im \Psi(2)) +
\exp(-(s+\si)\Re \Psi(2)) \cos((s+\si) \Im \Psi(2))\\
&\qquad \qquad \qquad +2  \exp(-s \Re \Psi(2)) \cos(s \Im \Psi(2)) +2 \exp(-\si \Re \Psi(2))\cos(\si \Im \Psi(2)) + 2\Big\}\\
&\qquad \qquad - \frac{a d}{4} \Big\{ \exp(-|s-\si| \Re \Psi(2))\cos(|s-\si| \Im \Psi(2)) + 
\exp(-(s+\si)\Psi(2))\cos((s+\si) \Im \Psi(2)) \\
&\qquad \qquad \qquad +2  \exp(-s \Re \Psi(2)) \cos(s \Im \Psi(2)) -2 \exp(-\si \Re \Psi(2)) \cos(s\Im \Psi(2))- 2\Big\}
\\
&\qquad \qquad + \frac{(b+c)^2}{8}\Big\{\exp(-|s-\si| \Re \Psi(2))\cos(|s-\si|\Im \Psi(2))-
\exp(-(s+\si) \Re \Psi(2)) \cos((s+\si) \Im \Psi(2))\Big\}\\
&\qquad \qquad + \frac{a (b+c)}{4} \Big\{\exp(-|s-\si|\Re \Psi(2)) \sin(-|s-\si| \Im \Psi(2))
+\exp((s+\si) \Re \Psi(2)) \sin((s+\si) \Im \Psi(2))\\
&\qquad \qquad \qquad - 2 \exp(-\si \Psi(2)) \sin(\si\Im \Psi(2)) \Big\}\\
&\qquad \qquad + \frac{d(b+c)}{4} \Big\{\exp(-|s-\si|\Re\Psi(2)) \sin(|s-\si|\Im \Psi(2)) -\exp(-(s+\si) \Re \Psi(2)) \sin((s+\si) \Im \Psi(2)) \\
&\qquad \qquad \qquad - 2\exp(-\si \Re \Psi(2)) \sin(\si \Im \Psi(2))  \Big\} \;d\si ds.
\end{align*}
Since all integrals over the exponential terms converge, 
their contribution in the preceding sum vanishes as $t\ra \infty$ as these are divided by $t^2$ 
and only the constants under the integrals survive. 
Therefore 
\begin{align*}
E_t \stackrel{t\ra \infty}{\lra} \frac{a^2}{4} + \frac{d^2}{4} + \frac{ad}{2} = \frac{(a+d)^2}{4}
\end{align*}
such that 
\begin{align*}
\EE\Big[\Big|\frac{1}{t}\int_0^t \pi_r K(Z_s) ds- \frac{a+d}{2}r\Big|^2\Big] 
&\stackrel{t\ra\infty}{\lra} \frac{(a+d)^2}{4} - \frac{(a+d)^2}{2} + \frac{(a+d)^2}{4} = 0.\\
\end{align*}

\section*{Acknowledgement}

The authors would like to thank the two anonymous referees for their thorough work, 
which has considerably improved both the presentation and the content of this article. 
Both authors thank Prof. Sylvie Roelly and the probability group of Universit\"at Potsdam 
for the hospitality during the one year stay of Paulo-Henrique da Costa in the framework 
of a postdoctoral 12 months stay in 2014 funded by the Ci\^encia sem fronteiras program of the Brazilian government by the grant CsF-CAPES/11786-13-2.  
The authors express their gratitude to the International Research Training Group Berlin - S\~ao Paulo: Dynamical Phenomena of Complex Networks 
for various infrastructure support and Prof. Paulo Ruffino for many inspiring discussions. 
The first author thanks Universidad de los Andes, School of Sciences, 
for the FAPA grant ``Stochastic dynamics of L\'evy driven systems'', 
which supported a visit of the second author.

\end{document}